\documentclass[journal=jacsat,manuscript=article]{achemso}

\usepackage{chemformula} 
\usepackage[T1]{fontenc} 

\usepackage{amsmath,amssymb,amsfonts}
\usepackage{graphicx}
\usepackage[english]{babel}
\usepackage{amsthm}
\usepackage{enumitem}

\newtheorem{remark}{Remark}

\newtheorem{theorem}{Theorem}

\newtheorem{lemma}[theorem]{Lemma}
\usepackage{multirow}
\usepackage{easyReview}
\renewcommand{\alert}[1]{#1}
\usepackage{caption}
\usepackage{subcaption}

\newcommand{\Real}{\mathbb R}

\SectionNumbersOn
\author{Chenchen Zhou}
\author{Hongxin Su}
\author{Xinhui Tang}
\author{Yi Cao}
\author{Shuang-hua Yang}
\email{yangsh@zju.edu.cn}
\affiliation[Zhejiang University]
{College of Chemical and Biological Engineering, Zhejiang University, Hangzhou, Zhejiang P. R. China}
\alsoaffiliation[quzhou]
{Institute of Zhejiang University-Quzhou, Quzhou, Zhejiang Province,  P.R. China}

\title[An \textsf{achemso} demo]
  {Global Self-Optimizing Control of Batch Processes}

\abbreviations{Re-gSOC}

\begin{document}

%
%
%
%
%

\begin{abstract}
	 This work considers to achieve near-optimal operation for a class of batch processes by employing self-optimizing control (SOC). Comparing with a continuous one, a batch process exhibits stronger nonlinearity with dynamics because of the non-steady operation condition. This necessitates a global version of SOC to achieve satisfactroy performance. Meanwhile, it also makes the existing global SOC (gSOC) not directly applicable to batch processes due to the causality amongst variables. Therefore, it is necessary to extend the original gSOC to batch processes. In addition to the nonconvexity challenge of the original gSOC problem, the new extension for batch processes has to face even more challenges. Particularly, the causality due to dynamics of batch processes brings in structural constraints on controlled variables (CVs), making a CV selection problem even more difficult. To address these challenges, the gSOC problem is recast in a vectorized formulation and it is proved that the structural constraints considered are linear in the vectorized formulation. Moreover, a novel shortcut method is proposed to efficiently find sub-optimal but more transparent solutions for this problem. 
	 The effectiveness of the new approach is validated through a case study of a fed-batch reactor, where CVs are constructed through a combination matrix with a repetitive structure, resulting in a simple SOC scheme. This simplicity facilitates the implementation of the SOC approach and enhances its practical applicability and robustness.

\end{abstract}

\pagebreak
\section{Introduction}

Batch and semi-batch processes are prevalent in the contemporary chemical industry for low-volume, but high-value-added products such as fine chemicals and polymers. The growing demand for product diversity as the market digitalizing has led to increased attention on the optimization of flexible operation of batch processes. 
However, the strictly high-quality standards associated with high-value-added products together with wide uncertainties in flexible operation conditions make the problem extremely challenging\cite{Yoo2021}.
Besides, the inherent dynamic nature of batch processes also increases the difficulty of achieving optimal operation. In most batch processes, a steady-state is never reached. This indicates that the trajectory of a batch process may span across a wide operational space, which means not only dynamics but also nonlinearities have to be taken into account for batch process control and optimization.

There are two main types of online optimization technologies for batch processes, namely run-to-run/batch-to-batch optimization and within-batch optimization. The former is based on the repetitiveness of the batch process to iteratively update batch-to-batch operations. Optimization actions for the next batch are carried out upon the completion of the previous batch, aiming to achieve incremental performance improvements throughout the iterations.
Iterative Learning Control (ILC) is a widely recognized method in this field.\cite{bristow2006,wang2021}
However, a shortcoming of this type of approaches is its failure to account for uncertainties that arise within a single batch process. \cite{zhang2004} As a result of such uncertainties, the behavior of the same process can vary from one batch to another, rendering knowledge learned from one batch inapplicable to another. This limitation makes ILC approaches unsuitable for such processes.

The latter type of optimization involves executing the optimizing control during batch operations to address the uncertainties that arise within the current batch period. Two types of schemes are commonly used for within-batch optimization: explicit schemes and implicit schemes.\cite{srinivasan2003}. Explicit schemes follow a two-step approach: firstly, parameter estimation is used to update the model, and then the updated optimization problem is repeatedly solved online to update the remaining batch operation.\cite{bristow2006}
One limitation of this approach is its reliance on the accuracy of the model, as a failure to update the model can potentially compromise the performance. However, conflicts may arise between model parameter estimation and process optimization.
Parameter estimation requires a persistent excitation, implying that inputs need to sufficiently vary to uncover unknown parameters. This requirement is often unacceptable when inputs are near optimal \cite{bristow2006}.
On the other hand, implicit schemes utilize measurements to directly adapt the inputs. Necessary Conditions of Optimality Tracking (NCO Tracking) \cite{srinivasan2003} technology is a widely utilized technology for batch process optimization.
NCO tracking has to deal with active constraints and plant economic sensitivities together. Active constraints need to be held constant and sensitivities within the remaining degrees of freedom need to be maintained at zero. {However, the sensitivities, typically gradients of a cost function, have to be evaluated online, which can be time-consuming \cite{Ye2022b}.}
Another promising method in implicit schemes is self-optimizing control (SOC).\cite{Skogestad2000, Jaschke2017}
SOC aims to achieve satisfactory performance by selecting controlled variables (CVs) off-line but maintaining them at constant set-points online. It effectively translates economic objectives into process control objectives, enabling regulatory feedback controllers to restore operational optimality.\cite{Morari1980}
In contrast to NCO, SOC indirectly achieves optimality through feedback control of CVs selected off-line \cite{gros2009a}, eliminating the need for online gradient evaluations, hence resulting in not only a significantly lower online computation load, but also a much quicker convergence to the (sub)optimum without multiple gradient evaluation and tracking cycles.
By capitalizing on the reduced online computation load and quicker convergence to the (sub)optimum, SOC offers a compelling solution for achieving optimal operation in batch processes.

Recently, the SOC for continuous processes has been extended to batch ones~\cite{Oliveira2016, Ye2017, Ye2022b}. In order to fit with the static formulation, the extended SOC spans a higher dimensional variable space with the step sequence of a batch process, leading to a static mapping between inputs and the objective function. Based on this, although the static SOC methodology can be applied, some structural constraints on the combination matrix have to be considered due to causality between variables. These structural constraints make the existing global SOC (gSOC) formulation not applicable, hence all these extended batch SOC methods have to be based on local SOC formulations. 
The inherent local limitation of these methods prevents them to appropriately address nonlinearities in batch processes. Without reaching a steady-state, the whole trajectory of a batch process may go across a wide operation space. Additionally, the dynamics of batch processes make trajectories highly sensitive to historical conditions. Even a minor change in the past condition could significantly alter the future trajectory. Although these methods adopt a trajectory-wise linearization, it is insufficient to account for the sensitivity to historical uncertainties, leading to unsatisfactory performance.

Motivated by the aforementioned shortcomings, in this paper, a SOC solution for batch processes is developed via reformulating the global SOC (gSOC) problem for batch processes. 
The gSOC problem was originally proposed for continuous processes\cite{Ye2015, Ye2013a, su2022, Ye2022} to minimize the average loss over the entire operating space. The problem was solved by using nonlinear simulation data over the entire operation space to select SOC CVs instead of using linear models, making it distinguished {from} local SOC methods. This paper is the first to consider the gSOC problem for batch processes.

The native gSOC problem is non-convex, {\em i.e.}, it is practically impossible to find the true global solution of the problem. As a concequence, it is difficult to make a meaningful comparison of different CVs. To address this issue, \citet{Ye2015} proposed a shortcut method that simplifies the gSOC problem by reformulating it as a quadratic programming problem for continuous processes. However, when dealing with batch processes, additional structural constraints on the combination matrix arise due to causality issues, further complicating the solution of the gSOC problem. Unfortunately, the existing shortcut methods in gSOC are not capable of handling such structural constraints effectively. 

It is worth to note that structural constraints on the combination matrix are common in SOC. In practice, sometimes different CVs composed of different measurements are wished, resulting in specific zeros in the combination matrix. \cite{jafariConvexReformulationsSelfoptimizing2022} 
To address the issue of structural constraints in SOC, \citet{jafariConvexReformulationsSelfoptimizing2022} introduced a Linear Matrix Inequality approach, which can handle various types of matrix structures, including block diagonal, triangular and structured measurement selection, involving specific zero elements in combination matrix.
However, when dealing with batch processes, the structural constraints become significantly more complex, necessitating the consideration of both diagonal and repetitive blocks to achieve simplicity and robustness in SOC implementation \cite{Ye2022b}. Regrettably, as of now, there are no methods available in the existing literature to specifically address these complex structural constraints \cite{Ye2018}. The difficulty of incorporating structural constraints likely prevented the original proposers of gSOC from adopting it for batch processes, leading them to use local SOC instead.\cite{Ye2018}
In all existing batch process SOC methods, the local formulation and direct numerical optimization are used to deal with the structural constraints.

Motivated by the aforementioned challenges, this paper revisits the gSOC problem and proposes a novel convex approximation method for the gSOC problem with structural constraints, which is also applicable to batch processes.

To the end, the paper presents 3 main contributions:
\begin{enumerate}
	\item Extend the gSOC method to address the dynamic optimization problem of batch processes.
	\item Prove that a class of structural constraints of CVs can be expressed as linear constraints in SOC problems.
	\item Propose a new shortcut algorithm for the gSOC problem.
\end{enumerate}

The structure of the rest paper is as follows:
Section 2 provides an overview of the main existing results for static gSOC methods, including a discussion of a shortcut method with convex approximation. 
In Section 3, the dynamic batch SOC problem is formulated as a static one with structural constraints. It is proven that a class of structural constraints can be expressed as linear equality constraints. And a novel convex approximation algorithm to solve batch gSOC with structural constraints is derived. 
Section 4 focuses on the investigation of a fed-batch reactor as a case study of the proposed approaches.
Finally, the paper concludes in the last section.
\subsection{Nation}
	 Let $I_m$ denote the $m \times m$ identity matrix. Use $\mathbf{0}_{m\times n}$ to represent a matrix of zeros, with subscript dimensions omitted when clear from context for simplicity. 
\section{Static gSOC}
\label{sec:gSOC}
To gain a deeper understanding of the challenges and solutions pertaining to batch gSOC, this section presents a concise overview of the existing static gSOC methodologies.

\subsection{Global approximation of economic loss}
Considering a typical chemical process with uncertainties, its operation optimization problem can be stated as follows:
\begin{equation}  
	\label{eq:sopt}
	\begin{aligned}
		&\min _{\mathbf{v}} J(\mathbf{v}, \mathbf{d}) \\
		&\text { s.t. }  \mathbf{g}(\mathbf{v}, \mathbf{d}) \leq 0 \\
	\end{aligned}
\end{equation}
where $J: \Real^{n_{v}} \times \Real^{n_{d}} \mapsto \Real$ is the steady state economic cost function, $ \mathbf{d} \in \Real^{n_d} $ stands for disturbances, and $ \mathbf{v} \in \Real^{n_v}$ stands for the steady state degrees of freedom, 
$\mathbf{g}: \Real^{n_{v}} \times \Real^{n_{d}} \mapsto \Real^{n_{g}}$ denotes the operational constraints.

Assuming that at an optimal point, {\em i.e.} a solution point of \eqref{eq:sopt}, a subset of $n_a$ constraints, $\mathbf{g}_a\in \mathbf{g}$ are bounded at zero. As the lower bounds are active, $\mathbf{g}_a:\Real^{n_a}$ are referred to as active constraints. Furthermore, it is assumed that the active constraints set $\mathbf{g}_a$ does not change within the whole operational space. Taking out these active constraints, the remaining degrees of freedom within the whole operational space reduces to $\mathbf{u}\in\Real^{n_u}$, where $n_u=n_v-n_a$. Correspondingly, the optimization problem becomes unconstrained with the reduced degrees of freedom, $\mathbf{u}$ as follows:
\begin{equation}
	\label{eq:usopt}
	\min _{\mathbf{u}} J(\mathbf{u}, \mathbf{d}) \\
\end{equation}

If all the disturbances $\mathbf{d}$ could be perfectly known, Problem~\eqref{eq:usopt} could be solved, and the solution $\mathbf{u}^{*}$ could be applied to get the optimal operation. Unfortunately, in practice, the perfect knowledge of disturbances is not available and such a strategy is not implementable. To counteract unknown disturbances, process measurements have to be used particularly through feedback control because these measurements contain certain information of disturbances, which can be represented in a measurement model as follows:
\begin{align}
	&\mathbf{y} = \mathbf{m}(\mathbf{u}, \mathbf{d}) \\
	&\mathbf{y}_{m} = \mathbf{y} + \mathbf{n}_y 
\end{align}
where $\mathbf{m}: \Real^{n_{u}} \times \Real^{n_{d}} \mapsto \Real^{n_{y}}$ describes the dependence of system output $ \mathbf{y} $ on $ \mathbf{u} $ and $ \mathbf{d} $, and finally,
$ \mathbf{y}_{m} \in \Real^{n_{y}} $ denotes the real measured signals and {$ \mathbf{n}_y \in \Real^{n_{y}}  $} denotes measurement noises. 

Among many efforts to restore process optimality through measurement feedback, the self-optimizing control (SOC) is the most promising one. It aims to find the most appropriate controlled variables (CVs) through an off-line analysis and when these CVs are maintained online at constant setpoints through feedback control, the steady state process operation is optimal or near optimal even under unknown disturbances. To achieve this goal, CV selection is the core. In general, a CV can be a parameterized function of process measurements, {\em i.e.} $\mathbf{c}=\mathbf{h}(\mathbf{y}_m,\boldsymbol{\theta})$ where $\boldsymbol{\theta}$ denotes the parameters. In practice, a linear combination CV is often sufficient to restory process optimality, {\em i.e.} $\mathbf{c} = \mathbf{H}\mathbf{y}_{m}$, where $\mathbf{H} \in \mathbb{R}^{n_{u} \times n_{y}}$ is the combination matrix.

An economic loss is commonly used to quantify the performance of CVs shown as follows:
\begin{equation}
	\label{eq:loss}
	L =J(\mathbf{u}|_{\mathbf{c}}, \mathbf{d})-J^{*}(\mathbf{d}) = L(\mathbf{d}, \mathbf{n}_y, \mathbf{H})
\end{equation}
where $ J^{*}(\mathbf{d}) $ stands for the optimal cost at a given disturbance $\mathbf{d}$, whilst $\mathbf{u}|_{\mathbf{c}}$ represent closed-loop $\mathbf{u}$ under selected CV $\mathbf{c}$. Therefore, the loss depends on not only disturbance $\mathbf{d}$, but also measurement noise, $\mathbf{n}_y$, as well as the CV combination matrix, $\mathbf{H}$. Instead of the loss at a single nominal point, the global SOC (gSOC) considers the
average loss in the whole uncertain space, spanned by $\mathbf d$ and $\mathbf n_y$, which can be derived as follows for CV selection:
\begin{equation}
	\label{eq:ELoss}
	\begin{aligned}
		L_{\mathrm{gav}}(\mathbf{H})&=\underset{d \in \mathcal{D}, \mathbf{n}_y \in \mathcal{N}}E[L(\mathbf{d}, \mathbf{n}_y, \mathbf{H})] \\  
		&=\int_{\mathbf{d} \in \mathcal{D}, \mathbf{n}_y \in \mathcal{N}} \rho(\mathbf{d}) \rho(\mathbf{n}_y) L(\mathbf{d}, \mathbf{n}_y, \mathbf{H}) \mathrm{d} \mathbf{n}_y \mathrm{d} \mathbf{d}\\
		&\approx \dfrac{1}{N}\sum_{i=1}^{N} L(\mathbf{d}_i, \mathbf{n}_{y,i}, \mathbf{H})
	\end{aligned}
\end{equation}
where $ E[.] $ and $ \rho(.) $ denote the expected value and the probability density of a random variable respectively. $ \mathcal{D} $ and $ \mathcal{N} $ are the variation region spanned by $ \mathbf{d} $ and $ \mathbf{n}_y  $, respectively. $\mathbf d_i$ and $\mathbf n_{y,i}$ are Monte Carlo samples over the whole uncertain space, and $N$ is the number of samples.

{To} sum up, the gSOC problem can be formulated as follows:
\begin{equation}  
	\label{eq:soc}
	\begin{aligned}
		&\min _{\mathbf{H}} L_{\mathrm{gav}}(\mathbf{H}) \\
		&\text { s.t. }  
		\mathbf{y} = \mathbf{m}(\mathbf{u}, \mathbf{d}) \\
		& \qquad \mathbf{y}_{m} = \mathbf{y} + \mathbf{n}_y \\
		& \qquad \mathbf{H}\mathbf{y}_m  =\mathbf{c}_s
	\end{aligned}
\end{equation}
where $\mathbf c_s$ is the setpoint vector of CVs.
It should be noted that the measurement vector $\mathbf{y}_m$ can be augmented to $\tilde{\mathbf{y}}_m = \left[ 1 \quad \mathbf{y}_m^{\top} \right]^{\top}$ such that the combination matrix $\mathbf{H}$ is augmented to $\tilde{\mathbf{H}} = \left[-\mathbf{c}_s \quad \mathbf{H}\right]$ and the augmented closed-loop condition is $\tilde{\mathbf{H}}\tilde{\mathbf y}=0$. For shack of convenience, in the rest of the paper, $\mathbf{ y}_m$ and $\mathbf H$ represent the augmented measurement vector and combination matrix.

Due to the inclusion of the process model in the loss, directly solving Problem~\eqref{eq:soc} is often difficult. To simplify the problem, quadratic approximation is considered \cite{Halvorsen2003}:
around the optimal points, applying the Taylor expansion to the economic loss \eqref{eq:loss} in terms of $\mathbf c$ leads to a quadratic form of the loss as follows:
\begin{equation}
	\label{eq:dlossdc}
	L \approx \dfrac{1}{2} \left(\mathbf{c} - \mathbf{c}^{*}\right)^{\top} \mathbf{J}_{c c}\left( \mathbf{c}- \mathbf{c}^{*}\right) 
\end{equation}
where $ \mathbf{c}^{*} $ denotes the optimal values of the CVs and $ \mathbf{J}_{c c} $ is the Hessian matrix of $ J $ with respect to $ \mathbf{c} $ evaluated at $ \mathbf{c}^{*} $. $ \mathbf{J}_{c c} $ could be represented as follows:
\begin{equation}
	\label{eq:jcc2juu}
	\mathbf{J}_{c c} = (\mathbf{H} \mathbf{G}_y)^{-\top}\mathbf{J}_{u u}(\mathbf{H} \mathbf{G}_y)^{-1}
\end{equation}
where $ \mathbf{J}_{u u} $ is the Hessian matrix of $ J $ with respect to $ \mathbf{u} $ evaluated at $ \mathbf{u}^* $, and $ \mathbf{G}_y $ is the sensitivity matrix of $ \mathbf{y} $ in terms of $ \mathbf{u} $. Since $ \mathbf{c} = \mathbf{H}\mathbf{y} = \mathbf{H}(\mathbf{y}_m - \mathbf{n}_y) $ and the feedback result $ \mathbf{H}\mathbf{y}_m = 0 $, we could get $ \mathbf{c} = -\mathbf{H}\mathbf{n}_y $. Furthermore, the optimal values of the CVs should be obtained as $ \mathbf{c}^{*} = \mathbf{H}\mathbf{y}^{*}  $, where $ \mathbf{y}^{*} $ denotes the optimal value of measurements. Hence, Eq.~\eqref{eq:dlossdc} could be formulated as 
\begin{equation}
	\label{eq:dlossdc_y}
	L \approx \dfrac{1}{2} \left(\mathbf{H}\mathbf{y}^{*}+\mathbf{H}\mathbf{n}_y\right)^{\top} \mathbf{J}_{c c}\left( \mathbf{H}\mathbf{y}^{*}+\mathbf{H}\mathbf{n}_y\right)
\end{equation}
So the average loss in \eqref{eq:ELoss} can be approximated as follows:
\begin{equation}
	\label{eq:qsloss}
	L_{\mathrm{gav}}(\mathbf{H}) \approx E[ \dfrac{1}{2} \left(\mathbf{H}\mathbf{y}^{*}+\mathbf{H}\mathbf{n}_y\right)^{\top} \mathbf{J}_{c c}\left( \mathbf{H}\mathbf{y}^{*}+\mathbf{H}\mathbf{n}_y\right) ]
\end{equation}

Based on the aforementioned approximation of the loss function, we can derive the following lemma:

\begin{lemma}
	\label{pro:1}
	The average loss in \eqref{eq:qsloss} for a given $ \mathbf{H} $ could be decomposed as 
	\begin{equation}
		L_{\mathrm{gav}}(\mathbf{H}) = E[L^d] + E[L^n]
	\end{equation}
	where
	\begin{equation}
		\label{eq:LdLn}
		\begin{aligned}
			L^d &= \dfrac{1}{2} \mathbf{y}^{*\top}\mathbf{H}^{\top} \mathbf{J}_{c c}  \mathbf{H}\mathbf{y}^{*} \\
			L^n &= \dfrac{1}{2} \mathrm{tr}(\mathbf{W}_y^2\mathbf{H}^{\top} \mathbf{J}_{c c}  \mathbf{H})
		\end{aligned}	
	\end{equation}
	$ \mathrm{tr}(\cdot) $ denotes the trace of a matrix, and $ \mathbf{W}_y^2 = E(\mathbf{n}_y\mathbf{n}^{\top}_y) $ is diagonal provided that $ \mathbf{n}_y $ are mutually independent.
\end{lemma}
\begin{proof}
	See \cite{Ye2015}
\end{proof}
Lemma \ref{pro:1} indicates the average loss can be considered as two parts, one is caused by disturbance $ \mathbf{d} $, and the other is caused by measurement error and noise $ \mathbf{n}_y $.

\begin{lemma}
	\label{pro:QH}
	The value of $ L $ does not change when $ \mathbf{H} $ is premultiplied by any invertible matrix $ \mathbf{Q} $. 
\end{lemma}
\begin{proof}
	See \cite{Ye2015}.
\end{proof}

\begin{remark}
	\label{re:QH}
	Lemma \ref{pro:QH}  reflects that scaling a whole set of CVs has no impact on the loss $ L $.  This property allows the designer to consider additional requirements by imposing specific constraints on $\mathbf{H}$, as long as these constraints can be expressed by pre-multiplying $\mathbf{H}$ with any non-singular matrix $\mathbf{Q}$. 
	For instance, \citet{Alstad2007} proposed the constraint $\mathbf{H} \mathbf{G}_y = \mathbf{I}$ to achieve a decoupled steady-state response from $\mathbf{u}$ to $\mathbf{c}$. Similarly, \citet{Yelchuru2012} utilized the constraint $\mathbf{H} \mathbf{G}_y = \mathbf{J}_{u u}^{1/2}$, which effectively ensures $\mathbf{J}_{c c} = \mathbf{I}_{n_u}$ (where $\mathbf{I}_{n_u}$ denotes the identity matrix of dimension $n_u \times n_u$), simplifying the loss calculation.
	However, when dealing with an $\mathbf{H}$ matrix that has structural constraints, certain limitations exist on the structure of the pre-multiplier to preserve the validity of these structural constraints. This aspect will be further explored in subsequent discussions.
\end{remark}

Owing to the intricacies of nonlinear mapping functions, the global mean loss may be appraised through Monte Carlo simulation founded on its primordial definition as follows:
\begin{equation}
	\label{eq:qsloss1}
	L_{\mathrm{gav}}(\mathbf{H}) \approx \sum_{i=1}^{N} \dfrac{1}{2} \left(\mathbf{H}\mathbf{y}_i^{*}+\mathbf{H}\mathbf{n}_{y,i}\right)^{\top} \mathbf{J}_{c c,i}\left( \mathbf{H}\mathbf{y}_i^{*}+\mathbf{H}\mathbf{n}_{y,i}\right) 
\end{equation}
where the subscript $i$ stands for i-th sample.

\subsection{gSOC problem and a shortcut method}
\label{sec:gSOCsc}
Based on the above approximation, the gSOC problem in \eqref{eq:soc} can be simplified to :
	\begin{equation}
		\label{eq:gSOC_opt}
		\begin{aligned}
			&\min _{\mathbf{H}} \dfrac{1}{N}\sum_{i=1}^{N}L(\mathbf{d}_i, \mathbf{n}_{y,i}, \mathbf{H}) = \dfrac{1}{N}\sum_{i=1}^{N}\left[ L^d_i+L^n_i \right] \\
		\end{aligned}
	\end{equation}
where the suscript $i$ stands for the sample number.

As shown in Eq.~\eqref{eq:gSOC_opt}, the optimization problem is non-convex, hence, it is difficult to get the global minimum solution. If using numerical algorithms to solve this problem, due to non-convexity, the solution is initial value dependent, {\em i.e.}, different initial values may produce different local optimal solutions. Therefore, for the convenience of further analysis, this problem is further simplified into a convex form. 
The existing way to approximate the problem is by treating $\mathbf{J}_{cc,i}$ as a constant over all disturbance scenarios. This approximation leads to a quadratic cost function with linear equality constraints, hence it can be analytically solved, although sub-optimally.
Thus, the next, we will present this shortcut method of solving the global SOC problem:
\begin{equation}
	\label{eq:gSOC_opt_short}
	\begin{aligned}
		&\min _{\mathbf{H}} {\bar{L}}_{\mathrm{gav}}(\mathbf{H})
		\approx \dfrac{1}{2}\sum_{i=1}^{N}  \mathbf{y}_i^{*\top}\mathbf{H}^{\top}   \mathbf{H}\mathbf{y}_i^{*} +  \mathrm{tr}(\mathbf{W}_y^2\mathbf{H}^{\top}   \mathbf{H})\\
		&\text { s.t. }  \mathbf{H G}_{y,0}=\mathbf{J}_{\mathrm{uu},0}^{1 / 2}
	\end{aligned}
\end{equation}
Here, the $ 0 $th disturbance scenario $ \mathbf{d}_0 $ is selected as a particular reference point to form the constraint, and $ {\bar{L}}_{\mathrm{gav}}(\mathbf{H})  $ denotes the further approximated average loss under all uncertainties with an enforced condition  $ \mathbf{H G}_{y,0}=\mathbf{J}_{\mathrm{uu},0}^{1 / 2} $.

The analytical solution of Problem~\eqref{eq:gSOC_opt_short} is
\begin{equation}
	\mathbf{H}^{\top}=\left(\tilde{\mathbf{Y}}^{\top} \tilde{\mathbf{Y}}\right)^{-1} \mathbf{G}_{y,0}\left(\mathbf{G}_{y,0}^{\top}\left(\tilde{\mathbf{Y}}^{\top} \tilde{\mathbf{Y}}\right)^{-1} \mathbf{G}_{y,0}\right)^{-1} \mathbf{J}_{\mathrm{uu},0}^{1 / 2}
\end{equation}
where
\begin{equation}
	\mathbf{Y}=\left[\begin{array}{c}
		{\mathbf{y}^*_{1}}^{\top} \\
		{\mathbf{y}^*_{2}}^{\top} \\
		\vdots \\
		{\mathbf{y}^*_{N}}^{\top}
	\end{array}\right]=\left[\begin{array}{cccc}
		1 & \mathrm{y}_{1,1}^* & \cdots & y^*_{n_{y},1} \\
		1 & y^*_{1,2} & \cdots & y^*_{n_{y},2} \\
		\vdots & \vdots & & \vdots \\
		1 & y^*_{1,N} & \cdots & y^*_{n_{y},N}
	\end{array}\right]
\end{equation}
\begin{equation}
	\tilde{\mathbf{Y}}=
	\left[
	\begin{array}{c}
		\frac{1}{\sqrt{N}}\mathbf{Y} \\
		\mathbf{W}
	\end{array}
	\right]
\end{equation}
$\mathbf{W} = \mathrm{diag}(0,\mathbf{W}_y)$.
According to Lemma~\ref{pro:QH}, by choosing a transformation matrix $\mathbf{Q} = \left(\mathbf{G}_{y,0}^{\top}\left(\tilde{\mathbf{Y}}^{\top} \tilde{\mathbf{Y}}\right)^{-1} \mathbf{G}_{y,0}\right)\mathbf{J}_{\mathrm{uu},0}^{-1 / 2} $, a more concise equivalent expression is given as
\begin{equation}
	\mathbf{H}^{\top}=\left(\tilde{\mathbf{Y}}^{\top} \tilde{\mathbf{Y}}\right)^{-1} \mathbf{G}_{y,0}.
\end{equation}

Nevertheless, this simplification mentioned is very crude, because, in reality, both $ \mathbf{J}_{u u} $ and $ \mathbf{G}_{{{y}}}$ are not constant across the entire operating window. Since the constraint is only valid at a reference point $k$, this simplification is equivalent to linearize $\mathbf{J}_{uu}^{1/2}$ around the reference point, hence, similar to the local SOC consideration.

\section{Batch processes SOC}
\subsection{Dynamic optimization of batch processes}
In this paper, the batch SOC method is founded on the static perspective of a batch process. The explanation of this concept has been thoroughly explored in the literature on batch process optimization\cite{Srinivasan2007}. Hence in the following, a concise overview of the established notions will be provided.

In particular, the following type of batch process optimization is considered:
\begin{subequations}
	\label{eq:ocp}
	\begin{align}
		&\min _{\bar{\mathbf{u}}}
		J=\phi(\mathbf{x}(L))+\sum_{k=0}^{L-1} \psi^k(\mathbf{x}(k), \mathbf{u}(k))\\
		&\text { s.t. }  \mathbf{x}(k+1) = \mathbf{f}^{k}(\mathbf{x}(k),\mathbf{u}(k),{\mathbf{d}}(k)) \\
		&\qquad \mathbf{x}(0)=\mathbf{x}_0        \\                                              
		&\qquad \bar{\mathbf{u}} = \left[\mathbf{u}(0),\mathbf{u}(1),\dots,\mathbf{u}(L-1)\right]
	\end{align}
\end{subequations}
\\alpha{AAA u- different follows}where $ k \in [0,L] $ and $ k \in \mathbb{N} $. $ L $ is a finite terminal time. $ \mathbf{x}(k) \in \Real^{n_{x}}$ and $ \mathbf{u}(k ) \in \Real^{n_{u}}$ stand for, respectively, the state variables and control input at time instant $ k $. $ \mathbf{f}^{k}(.):\Real^{n_{x}+n_{u}+n_{d}}\to\Real^{n_{x}} $ denotes time varying model of process dynamic functions. $ {\mathbf{d}}(k) $ is the disturbance at time instant $ k $.
Furthermore, $ \phi $ is a scalar cost associated with the final state $ \mathbf{x}(L) $, and $ \psi^k $ is the contribution at time $ k $ to the integrated cost, which is allowed to be a time-varying function of states and inputs. $\bar{\mathbf{u}} $ represents the stacked control input.

The measurement model is as follows: 
\begin{equation}
	\begin{aligned}
		& \mathbf{y}(k) = \mathbf{g}^{k}(\mathbf{x}(k)) \\
		&\mathbf{y}_m(k) = \mathbf{g}^{k}(\mathbf{x}(k)) + \mathbf{n}_y(k)\\
		&\mathbf{u}_m(k) = \mathbf{u}(k) + \mathbf{n}_u(k) \\
	\end{aligned}
\end{equation}
$ \mathbf{y}(k) \in \Real^{n_{y}}$ and $ \mathbf{y}_m(k) \in \Real^{n_{y}}$ are true measurement and measured output variables with noise. 
$\mathbf{u}(k) \in \Real^{n_{u}}$ and $ \mathbf{u}_m(k) \in \Real^{n_{u}}$ are the true control input and the measured control input with implementation error.
$  \mathbf{g}^{k}(.):\Real^{n_{x}}\to\Real^{n_{y}} $ is the measurement model. $\mathbf{n}_y(k) \in \Real^{n_{y}}$ and $\mathbf{n}_u(k) \in \Real^{n_{u}}$ are the measurement errors(noise) and the implementation error.
{\begin{remark}
 	In this work, we considered the batch process to be unconstrained or convertible to an unconstrained problem. Active constraints are typically selected as controlled variables first, allowing constrained optimization to be transformed into unconstrained optimization. However, it should be noted that active constraint sets can change in dynamic optimization problems, which is more difficult to handle than steady-state active set changes. For instance, endpoint constraints require action ahead of time, and switching between different active path constraint sets must be done at the right time. Addressing these challenges of dynamic active set changes represents an interesting direction for future research.
\end{remark}}

By stacked states, the problem could be reformulated to a static optimizing problem as:
\begin{subequations}
		\begin{align}
		&\min _{\bar{\mathbf{u}}}
		J=\phi(\mathbf{x}({L}))+\sum_{k=0}^{L-1} \psi^k(\mathbf{x}(k), \mathbf{u}(k))\\
		&\text { s.t. }  \bar{\mathbf{x}} = \mathbf{f}(\bar{\mathbf{u}},\bar{{\mathbf{d}}}) \label{eq:x=fud}                                
	\end{align}
\end{subequations}
where $ \mathbf{f} $ is the reformulated static mapping function of states.
And
	$ \bar{\mathbf{y}} = \mathbf{g}(\bar{\mathbf{x}})  \label{eq:y=gx} $
denotes the reformulated static mapping function of outputs.
\begin{equation}
	\label{eq:dsoc}
	\begin{aligned}
		& \bar{\mathbf{u}} = \left[\mathbf{u}(0)^{\top},\mathbf{u}(1)^{\top},\dots,\mathbf{u}(L-1)^{\top}\right]^{\top} \in \mathbb{R}^{n_u L}\\
		& \bar{\mathbf{x}} = \left[\mathbf{x}(1)^{\top},\mathbf{x}(2)^{\top},\dots,\mathbf{x}(L)^{\top}\right]^{\top} \in \mathbb{R}^{n_x L}\\
		& \bar{\mathbf{y}} = \left[\mathbf{y}(0)^{\top},\mathbf{y}(1)^{\top},\dots,\mathbf{y}(L-1)^{\top}\right]^{\top} \in \mathbb{R}^{n_y L}\\  
		& \bar{{\mathbf{d}}} = \left[\mathbf{x}(0)^{\top},{\mathbf{d}}(0)^{\top},{\mathbf{d}}(1)^{\top},\dots,{\mathbf{d}}(L-1)^{\top}\right]^{\top}  \in \mathbb{R}^{n_x + n_d L}\\ 
		& \bar{\mathbf{n}}_y = \left[\mathbf{n}_y(0)^{\top},\mathbf{n}_y(1)^{\top},\dots,\mathbf{n}_y(L-1)^{\top}\right]^{\top} \in \mathbb{R}^{n_y L}\\
		& \bar{\mathbf{n}}_u = \left[\mathbf{n}_u(0)^{\top},\mathbf{n}_u(1)^{\top},\dots,\mathbf{n}_u(L-1)^{\top}\right]^{\top} \in \mathbb{R}^{n_u L}\\ 
	\end{aligned}
\end{equation}
are the stacked inputs, states, measurements and disturbances. Noted that the initial states are also considered as disturbances.   

By substituting Eq. \eqref{eq:x=fud} into Eq. \eqref{eq:y=gx}, the resulting equation can be expressed as: 
\begin{equation}
	\bar{\mathbf{y}} = \mathbf{g}(\mathbf{f}(\bar{\mathbf{u}},\bar{{\mathbf{d}}})).
\end{equation}
This equation represents a static input-output mapping of a batch process. Consequently, when examining batch processes from a batch-wise standpoint, their dynamic optimization can be interpreted as a static optimization problem.

However, there are two features that make batch processes distinct from static processes:
\begin{enumerate}
	\item 
	The measurement of the future cannot be obtained at the present moment. Therefore, {when} designing controllers or controlled variables, only current and past measurement information can be used, rather than all measurement information, while in steady-state optimization, it is often assumed that all measurement information is available.
	\item 
	In steady-state optimization, all manipulation variables influence one another, yet in batch process optimization, owing to causal effects, future control inputs do not impinge upon current ones. The sequential nature of production necessitates foresight in manipulating inputs, lest preceding actions undermine subsequent performance. Optimization must therefore have a dynamic and prospective outlook.
\end{enumerate}

\subsection{Scheme of Self-optimizing controlled variables}

The object of this paper is to identify an optimal combination matrix
$ \mathbf{H} $, such that economic loss is minimized by selecting
self-optimizing CVs, $ \mathbf{c} =\mathbf{H}\mathbf{y} $.
However, due to dynamic and causality, the self-optimizing CVs in batch process {are} distinguished form steady-state process.
Herein lie two salient contrasts between crafting self-optimizing controlled variables for batch and steady-state processes:
\begin{enumerate}
	\item The measured variables in the future cannot be obtained at present due to the influence of causality, {and} at the same time, both current and past measured variables can be used to design CVs. So in a batch process, the measured variables can be expressed as:
	$$
	\bar{\mathbf{y}}(k) = \left[\mathbf{y}(0)^{\top},\mathbf{y}(1)^{\top},\dots,\mathbf{y}(k)^{\top}\right]^{\top} \in \mathbb{R}^{n_y(k+1)  }.
	$$
	\item\alert{ The CVs design in steady state is usually independent of the controller design. But in the batch process, both need to be considered at the same time.
	A feasible input adaptation law (controller) can be derived based on the obtained CV functions, provided $ \textbf{u} $ is included in the measurements. 
	$
	\mathbf{\xi}(k) := [\mathbf{y}(k)^{\top},\mathbf{u}(k)^{\top}]^{\top} \in \mathbb{R}^{(n_y+n_u) }
	$
	denotes the extend measurement, and 
	$$
	\bar{\mathbf{\xi}}(k) = \left[ 1,  \mathbf{\xi}(0)^{\top},\mathbf{\xi}(1)^{\top},\dots,\mathbf{\xi}(k)^{\top}\right]^{\top} \in \mathbb{R}^{(n_y+n_u)(k+1)+1  }.
	$$
	Here the constant $1$ is also extended so that the setpoint of the controlled variable can be treated as 0 in the subsequent steps.
	The CVs associated with time instant $t = k$ are 
	$$
	\mathbf{c}(k) = \mathbf{H}(k)\bar{\mathbf{\xi}}(k) = \mathbf{0}_{n_u} \in \Real^{n_u}
	$$
	with setpoints $ \mathbf{c}_s(k) \in \mathbb{R}^{n_u}$. 
	 $\mathbf{H}(k)  := \left[\mathbf{c}_s(k) \ \mathbf{H}_0(k) \ldots \mathbf{H}_k(k)\right] \in \Real^{n_u \times \left((n_y+n_u)(k+1)+1\right)}$, where $\mathbf{H}_j(k) \in \Real^{ n_u \times (n_y+n_u)}$ represents the contribution of measurements at time $j$ to CVs at time $k$. 
	$$  \mathbf{c}(k)=\underbrace{\left[\begin{array}{llll} 
			[\mathbf{c}_s(k) & \mathbf{H}_{[0:k-1]}(k)] & \mathbf{H}_{y,k}(k) & \mathbf{H}_{u,k}(k)
		\end{array}\right]}_{\mathbf{H}(k)} \underbrace{\left[\begin{array}{c}
			\bar{\mathbf{\xi}}(k-1) \\ 
			\mathbf{y}(k) \\
			\mathbf{u}(k) 
		\end{array}\right]}_{\bar{\mathbf{\xi}}(k)} =\mathbf{0}_{n_u} $$
	\begin{equation}
		\label{eq:u=hy}
		\Rightarrow \mathbf{u}(k) = \mathbf{H}_{u,k}^{-1}(k)\left(-[\mathbf{c}_s(k) \quad \mathbf{H}_{[0:k-1]}(k)]\bar{\mathbf{\xi}}(k-1)- \mathbf{H}_{y,k}(k) \mathbf{y}(k)\right)
	\end{equation}
	where $\mathbf{H}_{[0:k-1]}(k) := \left[\mathbf{H}_0(k) \ldots \mathbf{H}_{k-1}(k)\right] \in \Real^{n_u \times (n_y+n_u)k}$,  
	$ \mathbf{H}_{u,k}(k) $ and $ \mathbf{H}_{y,k}(k) $ are submatrices of $ \mathbf{H}_k(k) $, $ \mathbf{H}_k(k)  = [\mathbf{H}_{y,k}(k) \ \mathbf{H}_{u,k}(k)]$. 
	In the following discussion, we will introduce an additional {constraint} to keep that $ \mathbf{H}_{u,k}(k) $ is mathematically reversible in Section \ref{sec:sc}. }
\end{enumerate}
In conclusion, due to the influence of causality, the measured variables in the future cannot be obtained at present. And the manipulated variables require explicit inclusion within controlled variables.
The following typical structures will be discussed:
\begin{enumerate}[label=\textit{Structure} \arabic*,itemindent=50 pt]
	\item Lower-block triangular (LBT)	matrix 
\alert{	$$
	\begin{aligned}
		& \bar{\mathbf{H}}=\left[\begin{array}{cccccc}
			\mathbf{c}_s(0) & \mathbf{H}_0(0) & \mathbf{0}_{ n_u \times {(n_u+n_y)}} & \mathbf{0}_{ n_u \times {(n_u+n_y)}} & \cdots & \mathbf{0}_{ n_u \times {(n_u+n_y)}} \\
			\mathbf{c}_s(1) & \mathbf{H}_0(1) & \mathbf{H}_1(1) & \mathbf{0}_{ n_u \times {(n_u+n_y)}} & \cdots & \mathbf{0}_{ n_u \times {(n_u+n_y)}} \\
			\vdots & \vdots & \vdots & \vdots & \ddots & \vdots \\
			\mathbf{c}_s(L-1) & \mathbf{H}_0(L-1) & \mathbf{H}_1(L-1) & \mathbf{H}_2(L-1) & \cdots & \mathbf{H}_{L-1}(L-1)
		\end{array}\right] \\
		& \text { written as } \mathrm{LBT}[\mathbf{H}(0), \mathbf{H}(1), \ldots, \mathbf{H}(L-1)].
	\end{aligned} 
	$$ }
	For brevity, we write the big combination matrix as $\bar{\mathbf { H }}=\mathrm{LBT}[\mathbf{H}(0), \ldots, \mathbf{H}({L}-1)]$, where $\mathbf{H}(k) = \left[\mathbf{c}_s(k) \ \mathbf{H}_0(k) \ldots \mathbf{H}_k(k)\right]$.
	
	This structure maximally utilizes the information of the measured variables, where the current CVs are related to all the measured variable values up to the present moment, and CVs vary at each time instant.
	\item A varying block diagonal 
\alert{	$$
	\bar{\mathbf{H}}=\left[\begin{array}{cccccc}
		\mathbf{c}_s(0) &\mathbf{H}_0(0) & \mathbf{0}_{ n_u \times {(n_u+n_y)}} & \mathbf{0}_{ n_u \times {(n_u+n_y)}} & \cdots & \mathbf{0}_{ n_u \times {(n_u+n_y)}} \\
		\mathbf{c}_s(1) &\mathbf{0}_{ n_u \times {(n_u+n_y)}} & \mathbf{H}_1(1) & \mathbf{0}_{ n_u \times {(n_u+n_y)}} & \cdots & \mathbf{0}_{ n_u \times {(n_u+n_y)}} \\
		\vdots & \vdots & \vdots & & \ddots & \vdots \\
		\mathbf{c}_s(L-1) &\mathbf{0}_{ n_u \times {(n_u+n_y)}} & \mathbf{0}_{ n_u \times {(n_u+n_y)}} & \mathbf{0}_{ n_u \times {(n_u+n_y)}} & \cdots & \mathbf{H}_{L-1}(L-1)
	\end{array}\right]
	$$
	which is written as $\left[\bar{\mathbf{c}}_s \quad \operatorname{diag}[\mathbf{H}_0(0), \ldots, \mathbf{H}_{L-1}(L-1)]\right]$. 
	 $$\bar{{\mathbf{c}}}_s = \left[{\mathbf{c}}_s(0)^{\top},{\mathbf{c}}_s(1)^{\top},\dots,{\mathbf{c}}_s(L-1)^{\top}\right]^{\top} \in \mathbb{R}^{n_u L}$$ 
	 are stacked setpoints.
	In this structure, the CVs only consider the current measured variables, and the CVs also differ at each time instant. }
	\item A constant block diagonal matrix 
\alert{	$$
	\bar{\mathbf{H}} =
	\left[\begin{array}{cccccc}
		\mathbf{c}_s & \mathbf{H} & \mathbf{0}_{ n_u \times {(n_u+n_y)}} & \mathbf{0}_{ n_u \times {(n_u+n_y)}} & \cdots & \mathbf{0}_{ n_u \times {(n_u+n_y)}} \\
		\mathbf{c}_s & \mathbf{0}_{ n_u \times {(n_u+n_y)}} & \mathbf{H} & \mathbf{0}_{ n_u \times {(n_u+n_y)}} & \cdots & \mathbf{0}_{ n_u \times {(n_u+n_y)}} \\
		\vdots & \vdots & \vdots & & \ddots & \vdots \\
		\mathbf{c}_s & \mathbf{0}_{ n_u \times {(n_u+n_y)}} & \mathbf{0}_{ n_u \times {(n_u+n_y)}} & \mathbf{0}_{ n_u \times {(n_u+n_y)}} & \cdots & \mathbf{H}
	\end{array}\right] 
	$$
	where $ \mathbf{H} \in \Real^{ n_u \times {(n_u+n_y)}}$ is a time-invariant sub-matrix. This time, the individual CVs are
		$$
		\mathbf{c}(k) = \mathbf{H} \mathbf{\xi}(k) + \mathbf{c}_s
		$$
	As for this structure, it also only considers the current measured {variables} information, and for all time instances, the same CVs are applied with a shared set of parameters.}
	\end{enumerate}

The above CVs schemes could be considered as that the linear combination matrix of CVs has structural constraints \cite{Ye2018,Ye2022b}.

%
\subsection{Linear Representation of a Class of Structured Constraints}

In traditional SOC, such structural {constraints} (decentralized and triangular structures) are considered as nonlinear {constraints}\cite{Yelchuru2012}. 
Although some existing methods can address structural constraints for SOC in continuous processes \cite{Yelchuru2012,jafariConvexReformulationsSelfoptimizing2022}, handling structural constraints in batch processes is more challenging. This is because both diagonal and repetitive blocks in $\bar{\mathbf{H}}$ (Structure 3) are needed to enable simplicity and robustness of SOC \cite{Ye2022b}. Unfortunately, current techniques cannot handle these complex structural constraints for batch processes \cite{Ye2018}.

However, this work finds that a class of structural constraints are actually linear. As shown below, the linearity of these structural constraints becomes clear when represented in the vectorized form of $\bar{\mathbf{H}}$. 

Next, we introduce a general linear combination matrix $\mathbf{M} \in \mathbb{R}^{l \times w}$ for illustration. The vectorized form of $\mathbf{M}$ is defined as:
$$\operatorname {vec} (\mathbf{M}) = [h_{11}, h_{21}, \dots, h_{l1}, h_{12}, h_{22}, \dots, h_{l2}, \dots, h_{1w}, h_{2w}, \dots, h_{lw}]^{\top} \in \mathbb{R}^{lw}$$
where the columns of $\mathbf{M}$ are stacked one after another to form a column vector of length $lw$.

The structural constraints that may appear in batch process SOC problems can be categorized into two classes:	
{\begin{enumerate}
	\item 	specified elements, {\em i.e.} $h_{\alpha \beta}=b_{0}$, where $h_{\alpha \beta}$ is the $(\alpha,\beta)^{\mathrm{th}}$ element of $\mathbf M$ and $b_{0}$ is a specific constant, {\em e.g. $0$ or $1$}. 
	\item	repetitive elements, {\em i.e.} $h_{\alpha_1 \beta_1}=h_{\alpha_2 \beta_2}$, where $(\alpha_1 \beta_1)\neq (\alpha_2 \beta_2)$. 
\end{enumerate}}

Then, linearity of structural constraints is determined in the following lemma. 

\begin{lemma}
	\label{lem:QH=b}
	In $\operatorname {vec}(\mathbf{M})$, the above two classes of structural constraints are linear and can be expressed as follows:
	\begin{equation}
		\mathbf{Q}^{\top}\operatorname {vec} (\mathbf{M})=\mathbf{b}
	\end{equation}
	where $ \mathbf{Q} \in \Real^{{lw \times q }} $ is a matrix of $ \mathrm{rank}(\mathbf{Q})=q $ for total $q$ constraints, and $ \mathbf{b} \in \Real^{q} $ is a vector of known constants. $ \mathbf{Q}=\left[\mathbf{q}_1 \quad \mathbf{q}_2 \dots \mathbf{q}_q \right] $, $\mathbf{q}_j \in \mathbb{R}^{lw} $ and $ \mathbf{b}^{\top}=\left[b_1 \quad b_2 \dots b_q \right] $, ${b}_j \in \mathbb{R}$.
\end{lemma}
\begin{proof}
	
	For ease of expression, we introduce the vector $\mathbf{e}_m \in \mathbb{R}^n$, which represents a column vector with all elements equal to 0, except the {$m$}-th element which is 1.
	{For the first class, the structural {constraints} can be expressed as $\mathbf{e}_m^T\operatorname{vec}(\mathbf{M}) = b$, where $m = (\beta-1)\cdot n_u + \alpha$ and $b \in \mathbb{R}$ is a known constant.	
	For the second class, the structural constraints is given by $(\mathbf{e}_{m_1} - \mathbf{e}_{m_2})^T\operatorname{vec}(\mathbf{M}) = 0$, where  $m_1 = (\beta_1-1)\cdot n_u + \alpha_1$ and $m_2 = (\beta_2-1)\cdot n_u + \alpha_2$. }

\end{proof}

Next, we will use an example to illustrate how to construct $\mathbf{Q}$. This type of structural constraint was previously considered non-convertible to a linear constraint in the literature\cite{Yelchuru2012}.
As an example, considering a process with 2 inputs (degrees of freedom) and 3 measurements with 2 partially disjoint measurement sets \{1, 2\},\{2, 3\}, and $ h_{12}=h_{22} $; the structure is
$$
\mathbf{M}_I=\left[\begin{array}{ccc}
h_{11} & h_{12} & h_{13} \\
h_{21} & h_{22} & h_{23}
\end{array}\right].
$$
Here $ h_{13}=h_{21}=0 $ is the first type of structural constraints, and $ h_{12}=h_{22} $ is the second type. Three $ \mathbf{q}_j $ need to be constructed. 
The vectorized $ \mathbf{M}_I $ is expressed as
$$
\operatorname {vec} (\mathbf{M}_I) = \left[\begin{array}{cccccc}
h_{11} & h_{21} & h_{12} & h_{22} & h_{13} & h_{23}
\end{array}\right]^{\top}.
$$
The vector $ \mathbf{q}_j $ and the element $ b_j $ that need to be generated in order to make $ h_{13}=h_{21}=0 $ {are as follows:}
$$
\mathbf{q}_1 = \left[\begin{array}{cccccc}
0 & 0 & 0 & 0 & 1 & 0
\end{array}\right]^{\top}, \quad b_1=0
$$
and
$$
\mathbf{q}_2 = \left[\begin{array}{cccccc}
0 & 1 & 0 & 0 & 0 & 0
\end{array}\right]^{\top} , \quad b_2=0.
$$
The vector { $ \mathbf{q}_j $ and the element $ b_j $} that need to be generated in order to make {$ h_{12}=h_{22} $} {are as follows:}
$$
\mathbf{q}_3 = \left[\begin{array}{cccccc}
0 & 0 & 1 & -1 & 0 & 0
\end{array}\right]^{\top}, \quad b_3=0.
$$

\begin{remark}
There are some structural constraints {that} still can not be expressed as a linear form, for example, to select $m$ variables from total $n_y$ measurements, which is NP-hard, hence cannot be solved in linear program or quadratic program.\cite{caoBidirectionalBranchBound2008,kariwalaBidirectionalBranchBound2009,kariwalaBidirectionalBranchBound2010}
\end{remark}

In conclusion, by reformulating the SOC problem with the vectorized linear combination matrix, { the above two types} structural constraints are linear.

\subsection{Global SOC problem for batch process}
To proceed similarly to the static case, we can also use the economic loss {generated due to a deviation from} perfectly optimal operation as a means to evaluate the performance of the controlled variables (CVs). Around the optimal control sequence $\bar{\mathbf{u}}^* = \left[\mathbf{u}^*(0)^{\top}, \mathbf{u}^*(1)^{\top}, \dots, \mathbf{u}^*(L-1)^{\top}\right]^{\top}$, the economic loss can be evaluated in a quadratic form in terms of CVs deviations $\bar{\mathbf{c}}$, similar to Eq.~\eqref{eq:dlossdc_y}, as follows:
\alert{\begin{equation}
	\label{eq:dlossdbarc}
	L \approx \dfrac{1}{2} \left(\bar{\mathbf{H}}\bar{\mathbf{\xi}}^{*}+\bar{\mathbf{H}}\bar{\mathbf{n}}\right)^{\top} \mathbf{J}_{\bar{c} \bar{c}}\left( \bar{\mathbf{H}}\bar{\mathbf{\xi}}^{*}+\bar{\mathbf{H}}\bar{\mathbf{n}}\right)
\end{equation}
where 
\begin{equation}
	\label{eq:jcc2juu_}
	\mathbf{J}_{\bar{c} \bar{c}}= (\bar{\mathbf{H}} \bar{\mathbf{G}}_{\xi})^{-\top}\mathbf{J}_{\bar{u} \bar{u}}(\bar{\mathbf{H}} \bar{\mathbf{G}}_{\xi})^{{-1}}
\end{equation}
\begin{equation}
	\label{eq:nuny}
	\bar{\mathbf{n}} = \left[ 0,\mathbf{n}(0)^{\top},\mathbf{n}(1)^{\top},\dots,\mathbf{n}(L-1)^{\top}\right]^{\top} \in \mathbb{R}^{(n_y + n_u) L +1}   ,   \mathbf{n}(j) = [\mathbf{n}_y(j)^{\top} \ \mathbf{n}_u(j)^{\top}]^{\top}   
\end{equation}
$ \mathbf{J}_{\bar{u} \bar{u}} \triangleq  \dfrac{\text{d}^2{J}}{\text{d}{\bar{\mathbf{u}}^2}}$ and $\bar{\mathbf{G}}_{\xi} \triangleq  \dfrac{\text{d}{\bar{\mathbf{\xi}}}}{\text{d}{\bar{\mathbf{u}}}} $ are defined as the total derivative, {followed} by eliminating state variables in the dynamic process model.
$$
\bar{\mathbf{c}} = \left[{\mathbf{c}}(0)^{\top},{\mathbf{c}}(1)^{\top},\dots,{\mathbf{c}}(L-1)^{\top}\right]^{\top} \in \mathbb{R}^{n_u L}
$$ 
$$
\bar{\mathbf{\xi}}^{*} = \left[ 1,  \mathbf{\xi}^{*}(0)^{\top},\mathbf{\xi}^{*}(1)^{\top},\dots,\mathbf{\xi}^{*}(L-1)^{\top}\right]^{\top} \in \mathbb{R}^{(n_y+n_u)L+1  } \qquad,
\mathbf{\xi}^{*}(k)^{\top} \in \mathbb{R}^{(n_y+n_u)  }
$$
denote stacked CVs, stacked optimal extend measurements and optimal extend measurements, respectively.}

It is worth noting that $\bar{\mathbf{G}}_{\xi} \in \mathbb{R}^{((n_y+n_u)L+1) \times n_u L}$ is a block lower triangular matrix, it is expressed as follows:
\alert{\begin{equation}
	\bar{\mathbf{G}}_{\xi}=
	 \left[\begin{array}{ccccc}
		\mathbf{0}_{1\times n_u} & \mathbf{0}_{1\times n_u}& \mathbf{0}_{1\times n_u}& \dots & \mathbf{0}_{1\times n_u} \\
		 \bar{\mathbf{G}}_{\xi,(0,0)} & \mathbf{0}_{  {(n_u+n_y)}\times n_u} & \mathbf{0}_{  {(n_u+n_y)}\times n_u} & \cdots & \mathbf{0}_{  {(n_u+n_y)}\times n_u} \\
		 \bar{\mathbf{G}}_{\xi,(1,0)} & \bar{\mathbf{G}}_{\xi,(1,1)} & \mathbf{0}_{  {(n_u+n_y)}\times n_u} & \cdots & \mathbf{0}_{  {(n_u+n_y)}\times n_u} \\
		 \bar{\mathbf{G}}_{\xi,(2,0)}  & \bar{\mathbf{G}}_{\xi,(2,1)}  & \bar{\mathbf{G}}_{\xi,(2,2)} & \ddots & \vdots \\
	     \vdots & \vdots & \vdots & \ddots &\mathbf{0}_{  {(n_u+n_y)}\times n_u}\\
		 \bar{\mathbf{G}}_{\xi,(L-1,0)} & \bar{\mathbf{G}}_{\xi,(L-1,1)} & \bar{\mathbf{G}}_{\xi,(L-1,2)} & \cdots & \bar{\mathbf{G}}_{\xi,(L-1,L-1)}
	\end{array}\right]
\end{equation}
}
and each of its blocks $\bar{\mathbf{G}}_{\xi,(k_1,k_2)}$ can be expressed as: 
\begin{equation}
	\bar{\mathbf{G}}_{\xi,(k_1,k_2)} =\dfrac{\text{d}{\mathbf{\xi}(k_1)}}{\text{d}{\mathbf{u}(k_2)}}=
	\begin{cases}
		\left[\mathbf{0}_{n_u \times n_y } \quad \mathbf{I}_{n_u}	\right]^\top	\quad , k_1=k_2 \\
		\left[\dfrac{\text{d}{\mathbf{y}(k_1)}}{\text{d}{\mathbf{u}(k_2)}}^\top \quad \mathbf{0}_{n_u \times n_u}\right]^\top  \quad , k_1>k_2\\
		\mathbf{0}_{(n_u+n_y) \times n_u} \quad , k_1<k_2
	\end{cases} , k_1,k_2 = 0,1,2,\dots,L-1
\end{equation}
where $ \bar{\mathbf{G}}_{\xi,(k_1,k_2)} $ denotes the derivative of the measurement at the $k_1$-th time to the input at the $k_2$-th time.
$\dfrac{\text{d}{\bar{\mathbf{c}}}}{\text{d}{\bar{\mathbf{u}}}} =  \bar{\mathbf{H}} \bar{\mathbf{G}}_{\xi} $ is also a block lower triangular matrix, each of blocks can be expressed as:
\begin{equation}
	\dfrac{\text{d}{\bar{\mathbf{c}}}}{\text{d}{\bar{\mathbf{u}}}} \bigg|_{(k_1,k_2)}  =
	\begin{cases}
		\mathbf{H}_{u,k_2}(k_1)	\quad , k_1=k_2 \\
		\mathbf{H}_{y,k_2}(k_1)\dfrac{\text{d}{\mathbf{y}(k_1)}}{\text{d}{\mathbf{u}(k_2)}}  \quad , k_1>k_2\\
		\mathbf{0}_{n_u\times n_u} \quad , k_1<k_2
	\end{cases} , k_1,k_2 = 0,1,2,\dots,L-1
\end{equation}
where $\dfrac{\text{d}{\bar{\mathbf{c}}}}{\text{d}{\bar{\mathbf{u}}}} \bigg|_{(k_1,k_2)} $ denotes  the $(k_1,k_2)$ block of $\bar{\mathbf{H}} \bar{\mathbf{G}}_{\xi}$.

By utilizing the available results in the static case, the simplified batch process global dynamic SOC problem can be formulated as follows: 
\begin{equation}
	\label{eq:dgSOC_opt}
	\begin{aligned}
		&\min _{\bar{\mathbf{H}}} \dfrac{1}{N}\sum_{i=1}^{N}L(\bar{\mathbf{d}}_i, \bar{\mathbf{n}}_i, \bar{\mathbf{H}}) = \dfrac{1}{N}\sum_{i=1}^{N}\left[ L^d_i+L^n_i \right] \\
		&\text { s.t. } \bar{\mathbf{H}}\mbox{ on the form of a particular structure}.
	\end{aligned}
\end{equation}
where 
\begin{equation}
	\label{eq:batchLdLD}
	\begin{aligned}
		L^d_i &=\dfrac{1}{2} \left(\bar{\mathbf{H}}\bar{\mathbf{\xi}}^{*}_i\right)^{\top} \mathbf{J}_{\bar{c} \bar{c},i}\left(\bar{\mathbf{H}}\bar{\mathbf{\xi}}^{*}_i\right)\\
		L^n_i &= \dfrac{1}{2} \mathrm{tr}(\bar{\mathbf{W}}^2\bar{\mathbf{H}}^{\top}\mathbf{J}_{\bar{c} \bar{c},i}\bar{\mathbf{H}})
	\end{aligned}
\end{equation}
where $ {\bar{\mathbf{W}}^2} = diag(0,\mathbf{I}_{L} \otimes \mathbf{W}^2 )  $ and $ \mathbf{W}^2 = E(\mathbf{n}\mathbf{n}^{\top}) $. 
It is assumed that measurement noises of the same sensor at different time instances are independently and identically distributed (i.i.d.). 

Additionally, to enable a linear representation of structural constraints, $\mathbf{H}$ should be vectorized. To achieve this goal, we first need to introduce the Kronecker product and make use of the following properties associated with it.

\begin{lemma}\cite{strang2016}
	\label{lem:decouple}
	$ \mathbf{A} $, $ \mathbf{B} $, $ \mathbf{C} $ and $ \mathbf{X} $ are given matrices, and  $ \mathbf{A} \mathbf{X} \mathbf{B} = \mathbf{C} $, then
	\begin{equation}
		\left(\mathbf {B} ^{\top}\otimes \mathbf {A} \right)\,\operatorname {vec} (\mathbf {X} )=\operatorname {vec} (\mathbf {AXB} )=\operatorname {vec} (\mathbf {C} ).
	\end{equation}
	where $ \otimes $ denotes Kronecker product, and $  \operatorname {vec} (\mathbf {X} ) $ denotes  the vectorization of the matrix $ \mathbf {X}  $, formed by stacking the columns of $ \mathbf {X}  $ into a single column vector.
\end{lemma}
By applying Lemma \ref{lem:decouple} to Equation~\eqref{eq:batchLdLD}, we obtain the following result:
\begin{equation}
	\label{eq:LdLn_vec}
	\begin{aligned}
		L^d_i &= \dfrac{1}{2} \left\| \left(\bar{\mathbf{\xi}}_{i}^{*\top}\otimes \mathbf{J}_{\bar{c} \bar{c},i}^{1/2} \right) \operatorname {vec} (\mathbf{\bar{H}})   \right\|^2   \\
		L^n_i &=\frac{1}{2}	\operatorname{vec}(\mathbf{\bar{H}})^{\top}
		{\bar{\mathbf{W}}^2} \otimes \mathbf{J}_{\bar{c} \bar{c},{i} } 
		\operatorname{vec}(\mathbf{\bar{H}})\\
	\end{aligned}	
\end{equation}

The simplified dynamic self-optimizing control problem for batch processes based on the vectorized $\mathbf{\bar{H}}$ is:
{\begin{equation}
	\label{eq:dgSOC_opt_vecH}
	\begin{aligned}
		&\min _{\operatorname{vec}(\bar{\mathbf{H}})} \dfrac{1}{2N}\sum_{i=1}^{N} \left(\left\| \left(\bar{\mathbf{\xi}}_{i}^{*\top}\otimes \mathbf{J}_{\bar{c} \bar{c},i}^{1/2} \right) \operatorname {vec} (\mathbf{\bar{H}})   \right\|^2 +	\operatorname{vec}(\mathbf{\bar{H}})^{\top}
		{\bar{\mathbf{W}}^2} \otimes \mathbf{J}_{\bar{c} \bar{c},{i} } 
		\operatorname{vec}(\mathbf{\bar{H}}) \right)\\
		&\text { s.t. }  \mathbf{Q}_{str}^{\top}\operatorname {vec} (\bar{\mathbf{H}})=\mathbf{b}_{str}
	\end{aligned}
\end{equation}}
where $\mathbf{Q}_{str}^{\top}\operatorname {vec} (\bar{\mathbf{H}})=\mathbf{b}_{str}$ denotes the structural constraints and $\mathbf{Q}_{str} \in \mathbb{R}^{(n_uL+n_yL+1)n_uL \times q_{str}}$, $\mathbf{b}_{str} \in \mathbb{R}^{q_{str}}$ where $q_{str}$ is the number of the structural constraints.
At this point, gSOC has been extended to batch processes, and we can directly solve Problem \eqref{eq:dgSOC_opt} using existing numerical methods.

\subsection{A shortcut method of global dynamic SOC }
\label{sec:sc}
Like steady-state gSOC, the batch process problem is nonconvex with multiple local optima depending on initialization. For analytical convenience, an extra constraint $\mathbf{J}_{cc}=\mathbf{I}$ is typically imposed to convexify gSOC approximations. However, the inherently structural constraint in batch SOC, {\em e.g.} in Structures 2 and 3, conflicts with this extra shortcut constraint.
Therefore, the gSOC shortcut method in Section \ref{sec:gSOCsc} becomes invalid for batch process gSOC, necessitating a new approach. 
Per Lemma \ref{pro:QH}, optimal controlled variables are non-unique, so additional constraints are needed to improve solution stability without sacrificing optimality. At the same time, those constraints must not conflict with structural constraints. To propose the new shortcut used herein, we first derive the following Lemma :

\begin{lemma}
	\label{lem:diag}
	For the three structures considered in batch process SOC, there exists a diagonal and nonsingular matrix $\mathbf{\Lambda}$ and $\mathbf{\Lambda}$ is not an identity matrix, the structural constraint can be valid under a nonsingular transformation $\mathbf{H}'=\mathbf{\Lambda}\mathbf{H}$.
\end{lemma}
\begin{proof}
	$h_{\alpha\beta}$ denotes the $(\alpha,\beta)^{\mathrm{th}}$ element of $\mathbf H$ and $\lambda_{\alpha}$ denotes the $\alpha^{\mathrm{th}}$ diagonal element of $\mathbf{\Lambda}$.
	So the $(\alpha,\beta)^{\mathrm{th}}$ element of $\mathbf H '$ could be expressed as ${h'}_{\alpha\beta} =\lambda_{\alpha} h_{\alpha\beta}$.
	For structure 1 and 2, there exist structural constraints of specific positions with zero elements. Thus, for such architectures, left-multiplying by any nonsingular diagonal matrix preserves the validity of the structural constraints. That is if $h_{\alpha\beta}=\mathbf{0}_{ n_u \times {(n_u+n_y)}}$, ${h'}_{\alpha\beta} =\lambda_{\iota} h_{\alpha\beta} =0 $.
	
	For all three structures, left-multiplying by the matrix $\mu \mathbf{I}_{n_u}$ preserves the validity of the structural constraints, where $\mu$ is a nonzero constant, and $\mathbf{I}_{n_u} \in \Real^{n_u \times n_u}$ is an identity matrix.
\end{proof}


According to Lemma~\ref{lem:diag} and considering the characteristics of batch processes, we set the controlled variables in the following form:
\begin{equation}
	\mathbf{c}(k) = \left([\mathbf{c}_s(k) \quad \mathbf{H}_{[0:k-1]}(k)]\bar{\mathbf{\xi}}(k-1)+\mathbf{H}_{y,k}(k) \mathbf{y}(k)\right) - \mathbf{u}(k), k=0,1,...,L-1
\end{equation}
that is, with {$\mathbf{H}_{u,k}(k) = -\mathbf{I}_{n_u}$} fixed.
According to Lemma \ref{lem:QH=b}, this additional constraint can also be expressed as linear, which is denoted as follows:
\begin{equation}
	\mathbf{Q}_{add}^{\top}\operatorname {vec} (\bar{\mathbf{H}})=\mathbf{b}_{add}
\end{equation}
where $\mathbf{Q}_{add} \in \mathbb{R}^{(n_uL+n_yL+1)n_uL \times q_{add}}$, $\mathbf{b}_{add} \in \mathbb{R}^{q_{add}}$ where $q_{add} = n_u n_u L$.

As per Lemma \ref{pro:QH}, this does not compromise the performance of the controlled variables. Additionally, fixing {$\mathbf{H}_{u,k}(k)$} in this way ensures that $\mathbf{H}_{u,k}(k)$ is invertible, thereby allowing the control inputs $\mathbf{u}(k)$ to be recovered from the controlled variables.

Nevertheless, this simplification in Problem\eqref{eq:gSOC_opt_short} is very crude, since in reality, both $ \mathbf{J}_{u u} $ and $ \mathbf{G}_{{{y}}}$ are not constant across the operating window. Hence, a novel {and} gentler approximation is proposed herein:
\begin{equation}
	\mathbf{J}_{cc,i} \approx \mathbf{V}^{\top} \mathbf{J}_{u u,i} \mathbf{V}  \quad i = 1,2,\dots,N
\end{equation}
where  $\mathbf{V} \triangleq \mathbf{J}_{u u,0}^{-1/2} $. {The subscript $i$ denotes the $i$-th sample and $i=0$ denotes a nominal point sample which is pre-selected in advance.}
It is worth noting that the approximation used here is similar to the previous method in Problem~\eqref{eq:gSOC_opt_short}, but not identical. 
In Problem~\eqref{eq:gSOC_opt_short}, it relaxes one restriction by treating $\mathbf{J}_{cc}$ {as a} constant {matrix} over all disturbance scenarios, same as local SOC methods.
But in here, we just ignore the variation of $ (\mathbf{H} \mathbf{G}_{{y}})^{-1} $ which is one of the two nonlinear components of $ \mathbf{J}_{c c} $. $ \mathbf{J}_{u u} $, another nonlinear components of $ \mathbf{J}_{c c} $, is retained.

Accordingly, in batch process SOC, we approximate $J_{\bar{c}\bar{c}}$ as:
\begin{equation}
	\mathbf{J}_{\bar{c}\bar{c},i} \approx \mathbf{\bar{V}}^{\top} \mathbf{J}_{\bar{u} \bar{u},i} \mathbf{\bar{V}}  \quad i = 1,2,\dots,N
\end{equation}
where  $\mathbf{\bar{V}} \triangleq \mathbf{J}_{\bar{u} \bar{u},0}^{-1/2} $.

\alert{For convenience, we introduce the data matrix containing all sample data as
$$
\tilde{\mathbf{\Xi}} = \left[ \tilde{\mathbf{\xi}}_{1}^{\top} \ \tilde{\mathbf{\xi}}_{2}^{\top} \dots \tilde{\mathbf{\xi}}_{N}^{\top} \right]^{\top} \in \mathbb{R}^{ Nn_uL \times n_u L (1+n_u L+n_y L)}
$$
whose $i^\text{th}$ row block corresponds to
$$
\tilde{\mathbf{\xi}}_{i} =\bar{\mathbf{\xi}}_{i}^{*\top}\otimes \mathbf{J}_{\bar{u}\bar{u},i}^{1/2}\mathbf{\bar{V}} \in  \mathbb{R}^{ n_uL \times n_u L (1+n_u L+n_y L)}
$$
 at the optimal status under $\bar{{\mathbf{d}}}_i$. It can be verified that the $E(L^d)$ equals
\begin{equation}
	E(L^d) = \dfrac{1}{2N}\sum_{i=1}^{N} \left\| \left(\bar{\mathbf{\xi}}_{i}^{*\top}\otimes \mathbf{J}_{\bar{c} \bar{c},i}^{1/2} \right) \operatorname {vec} (\mathbf{\bar{H}})   \right\|^2 = \frac{1}{2N} \left\| \tilde{\mathbf{\Xi}} \operatorname {vec} (\bar{\mathbf{H}}) \right\|_{\mathrm{2}}^{2} .
\end{equation}}

\alert{Similarly, $E(L^n)$ is equivalent to:
$$E(L^n)  = \dfrac{1}{2N}\sum_{i=1}^{N}	\operatorname{vec}(\mathbf{\bar{H}})^{\top}
{\bar{\mathbf{W}}^2} \otimes \mathbf{J}_{\bar{c} \bar{c},{i} } 
\operatorname{vec}(\mathbf{\bar{H}}) =\frac{1}{2N}\left\| \breve{\mathbf{W}}{\operatorname {vec} (\bar{\mathbf{H}})}\right\|_{\mathrm{2}}^{2},
$$
where 
$$  \breve{\mathbf{W}} = \left( {\bar{\mathbf{W}}^2} \otimes \mathbf{\bar{V}}^{\top}\left(\sum_{i=1}^{N}\mathbf{J}_{\bar{u} \bar{u},i}\right)\mathbf{\bar{V}} \right)^{1/2} \in  \mathbb{R}^{n_u L (1+n_u L+n_y L)\times n_u L (1+n_u L+n_y L)} .$$}

\alert{In summary, we propose the following shortcut method for batch process global dynamic self-optimizing control(gdSOCsc): 
\begin{equation}
	\label{eq:re_gSOC_opt_batch}
	\begin{aligned}
		\min _{\bar{\mathbf{H}}} \overline{\bar{L}}_{\mathrm{gav}}(\bar{\mathbf{H}})
		&=\frac{1}{2N} \left\| \tilde{\mathbf{\Xi}} \operatorname {vec} (\bar{\mathbf{H}}) \right\|_{\mathrm{2}}^{2} +
		\frac{1}{2N}\left\| \breve{\mathbf{W}}{\operatorname {vec} (\bar{\mathbf{H}})}\right\|_{\mathrm{2}}^{2}\\
		&=\frac{1}{2N}\left\|\bar{\tilde{\mathbf{\Xi}}}\operatorname {vec} (\bar{\mathbf{H}})\right\|_{\mathrm{2}}^{2} \\
		&\text{s.t. }\quad \mathbf{Q}_{all}^{\top}\operatorname {vec} (\bar{\mathbf{H}})=\mathbf{b}_{all}
	\end{aligned} ,
\end{equation}}
where 
$$
\bar{\tilde{\mathbf{\Xi}}} = \left[
\begin{array}{c}
	{\tilde{\mathbf{\Xi}}} \\
	\breve{\mathbf{W}}
\end{array}
\right]
\quad
\mathbf{Q}_{all}=
\left[
\begin{array}{c}
	\mathbf{Q}_{str} \\
	\mathbf{Q}_{add}
\end{array}
\right]
\quad
\mathbf{b}_{all}=
\left[
\begin{array}{c}
	\mathbf{b}_{str} \\
	\mathbf{b}_{add}
\end{array}
\right] .
$$

The analytical solution of Problem~\eqref{eq:re_gSOC_opt_batch} is
\begin{equation}
	\operatorname{vec}(\bar{\mathbf{H}}) =  \left(\tilde{\mathbf{\Xi}}^{\top}\tilde{\mathbf{\Xi}}\right)^{-1}\mathbf{Q}_{all}\left(\mathbf{Q}_{all}^{\top}\left(\tilde{\mathbf{\Xi}}^{\top}\tilde{\mathbf{\Xi}}\right)^{-1}\mathbf{Q}_{all}\right)^{-1}\mathbf{b}_{all}.
\end{equation}

\section{Case studies}

\subsection{Process description}
The proposed gdSOC is applied to a fed-batch reactor, which comprises two reactions \cite{gros2009,Oliveira2016}: $ A + B \rightarrow C $ and $ 2B \rightarrow D $. Here, A and B represent the reactants, C is the desired product, and D is the byproduct generated by the side reaction.
A simple dynamical model for the system is derived by applying mass balance principles:
\begin{align}	
	\label{eq:PM1}	& \frac{d c_A}{d t}=-k_1 c_A c_B-c_A u / V, \quad c_A(0)=c_{A 0} \\
	\label{eq:PM2}	& \frac{d c_B}{d t}=-k_1 c_A c_B-2 k_2 c_B^2-\left(c_B-c_B^{i n}\right) u / V, \quad c_B(0)=c_{B 0} \\
	\label{eq:PM3}	& \frac{d V}{d t}=u, \quad V(0)=V_0 \\
	\label{eq:PM4}	& c_C=\left(c_{A 0} V_0-c_A V\right) / V \\
	\label{eq:PM5}	& c_D=\left(\left(c_A+c_{B i n}-c_B\right) V-\left(c_{A 0}+c_{B i n}-c_{B 0}\right)\right) / 2 V
\end{align}
where \(c_A\) represents the concentration of species A, \(c_B\) represents the concentration of species B, \(k_1\) and  \(k_2\) is the kinetic coefficient, \(u\) is the flow rate, and \(V\) is the volume. \(c_{A0}\) represents the initial concentration of species A.  \(c_B^{in}\) is the inlet concentration of B. where \(c_C\) represents the concentration of species C and \(c_D\) represents the concentration of species D.

\begin{table}[htbp]
	\centering
	\caption{Variable Description}
	\label{tbl:PM}
	\begin{tabular}{llll}
		\hline
		Variable & Description & Nominal Value  & Unit \\
		\hline
		$c_{A0}$ & initial concentration (A) & 0.72  & mol/l \\
		$c_{B0}$ & initial concentration (B) & 0.0614 &  mol/l \\
		$V_{0}$ & initial volume & 1 & l \\
		$k_{1}$ & kinetic coefficient (main) & 0.053 & l/(mol$\cdot$min) \\
		$k_{2}$ & kinetic coefficient (side) & 0.128 & l/(mol$\cdot$min) \\
		$c_{B}^{\text{in}}$ & inlet concentration of B & 5 & mol/l \\
		$t_{f}$ & batch duration & 250 & min \\
		\hline
	\end{tabular}
\end{table}
The operational objective is to maximize the amount of product C whilst minimizing the byproduct D at the final batch time \(t_f\), by manipulating the feed rate of reactant B, \(u(t)\), which is constrained within the bound \(0 \leq u(t) \leq 0.005\) [l/min]. The batch time is fixed at \(t_f = 250\) min. Therefore, a dynamic optimization problem is formulated as follows:
\begin{equation}
	\begin{aligned}
		&\min_{u(t)} J = [c_D(t_f) - c_C(t_f)]V(t_f) \\
		&\text{s.t.} \mbox{ process model \eqref{eq:PM1}-\eqref{eq:PM5}}\\
		&\quad \quad 0 \leq u(t) \leq 0.005
	\end{aligned}
\end{equation}

The process parameters, together with their nominal values, are given in Table~\ref{tbl:PM}. 
Firstly, by parameterizing the control vector with piecewise constant inputs and $L=20$, the problem is transformed into a static optimization problem: 
\begin{equation}
	\begin{aligned}
		&\min_{\bar{\mathbf{u}}} J = [c_D(L) - c_C(L)]V(L) \\
		&\text{s.t.} \ \bar{\mathbf{x}} = \mathbf{f}(\bar{\mathbf{u}},\bar{{\mathbf{d}}}) \\
		&\quad \quad 0 \leq u(k) \leq 0.005 \quad k = 0,2,\dots,L
	\end{aligned}
\end{equation}
where $\bar{\mathbf{u}} = \left[u(0) \ u(1) \dots \ u(L-1)\right]^{\top}$, $ \mathbf{x}(k) = \left[c_A(k) \ c_B(k) \ V(k)\right]^{\top} $ and $\bar{\mathbf{x}} = \left[\mathbf{x}(1)^{\top} \ \mathbf{x}(2)^{\top} \dots \ \mathbf{x}(L)^{\top}\right]^{\top}$.
$\mathbf{x}(k+1) = \mathbf{f}_k(\mathbf{x}(k),{u}(k),{\mathbf{d}}(k))$
and $ \mathbf{f}_k(\mathbf{x}(k),{u}(k),{\mathbf{d}}(k)) = \int_{k t_s}^{(k+1)t_s} f(\mathbf{x}(k),{u}(k),{\mathbf{d}}(k))d{t} $, where $f(.)$ denotes the process model\eqref{eq:PM1}-\eqref{eq:PM3}, $t_s=250/20 \text{ min}$.

The considered process measurements including the three system states are represented as:
\begin{equation}
	\mathbf{y}(k) = [c_A(k) \ c_B(k) \ V(k) ]^{\top}
\end{equation}
Gaussian noise with zero mean is assumed to be present in all measurements. The standard deviations for concentrations($ c_A(k) \ c_B(k) $) and volume($ V(k) $) are 0.03 mol/L and 0.1 L, respectively. 
To account for potential implementation errors in the manipulated variable, a small noise of 0.025 ml/min is included in the representation of $ u(t) $, enhancing the realism of the model.

In order to validate the proposed global dynamic self-optimizing control method, we designed the following 3 cases to test its applicability and effectiveness.
\begin{enumerate}[label=\textbf{Case} \arabic*,itemindent=50 pt]
	
	\item $\mathbf{d}=[c_{A0} \ c_{B0} \ V_{0} \ k_1 \ k_2]^{\top}$ with 10\% disturbances to compare the performance of different self-optimizing control methods under this disturbance.
	
	\item $\mathbf{d}=[c_{A0} \ c_{B0} \ V_{0} \ k_1 \ k_2]^{\top}$ with separately 5\%, 10\%, 20\% disturbances to compare the performance of various self-optimizing control methods under different disturbance levels.
	
	\item $\mathbf{d}=[c_{A0} \ c_{B0} \ V_{0} \ k_1 \ k_2]^{\top}$ with 10\% disturbances to obtain the self-optimizing controlled variables. Then two additional cases were tested to evaluate the controlled variable performance under unknown disturbances:	
	$\mathbf{d}=[c_{A0} \ c_{B0} \ V_{0} \ k_1 \ k_2]^{\top}$ with 20\% disturbances and
	$\mathbf{d}=[c_{A0} \ c_{B0} \ V_{0} \ k_1 \ k_2 \ c_{B}^{in}]^{\top}$ with 20\% disturbances.
	
\end{enumerate}

\subsection{Results and simulations}
\textbf{Case 1: Comparing global and local methods}
The initial values of the three states and the reaction kinetics parameters are uncertain and are defined as uniformly distributed fluctuations within $\pm{10\%}  $ of the nominal point.

Within the uncertainty range, Monte Carlo sampling is employed to randomly generate 200 scenarios. The first 100 samples are used for designing the controlled variable, while the remaining 100 samples are used for evaluating the performance of the controlled variable. 
Furthermore, for each scenario, $ \mathbf{J}_{\bar{u}\bar{u}} $ and $ \bar{\mathbf{G}}_{\xi} $ are calculated. In this case, we only consider Structure 3, which represents the time-invariant controlled variable, in order to compare the differences between different methods: local dynamic SOC(ldSOC)\cite{Ye2018}, global dynamic SOC(gdSOC) and shortcut method of global dynamic SOC(gdSOCsc).
Based on these three methods, the linear combination matrices $\mathbf{H}$ are given as follows:
\alert{$$
\mathbf{H}_{ld} = [ \  -1.1177\times 10^{-3} \  8.4840\times 10^{-3} \  -4.9812\times 10^{-4} \  -1], c_{s,ld} = 6.4777\times 10^{-5}
$$
$$
\mathbf{H}_{gd} = [ 7.4750\times 10^{-4} \  -3.1080\times 10^{-4} \  3.4359e\times 10^{-4} \  -1] , c_{s,gd} =  -1.8894\times 10^{-4}
$$
$$
\mathbf{H}_{gdsc} = [ 5.9681\times 10^{-4} \  8.3775\times 10^{-5} \  1.4724\times 10^{-4} \  -1] c_{s,gdsc} = 6.6631\times 10^{-5} 
$$}
where subscripts indicate different methods.
\begin{table}[h]
	\centering
	\caption{Average Loss Evaluations with 100 scenarios of Random Disturbances($\pm{10\%}  $) and Measurement Errors under Structure 3} 
	\begin{tabular}{cccc}
		\hline
		& Local       & \multicolumn{2}{c}{Global} \\
		& {ldSOC} & {gdSOC} & {gdSOCsc} \\
		\hline
		Structure 3 & 0.01018 &	0.00730 &	0.00743
		\\
		\hline
	\end{tabular}
	\label{tb:10Loss}
\end{table}
The average economic loss for the three methods across 100 test scenarios is shown in Table~\ref{tb:10Loss}.

The results in Table 2 demonstrate that the global SOC method (gdSOC) achieves substantially lower economic loss compared to the local SOC method (ldSOC). This significant difference highlights the importance of accounting for nonlinearity through a global approach when designing controlled variables for batch processes.
Furthermore, the comparable performance between the gdSOC method and the shortcut gdSOC method (gdSOCsc) provides justification for the approximations made in the shortcut approach. The shortcut allows analytical solutions to be obtained while retaining good optimization capability.

\begin{figure}[h]
	\centering

	\begin{subfigure}[b]{0.45\textwidth}
		\centering
		\includegraphics[width=\textwidth]{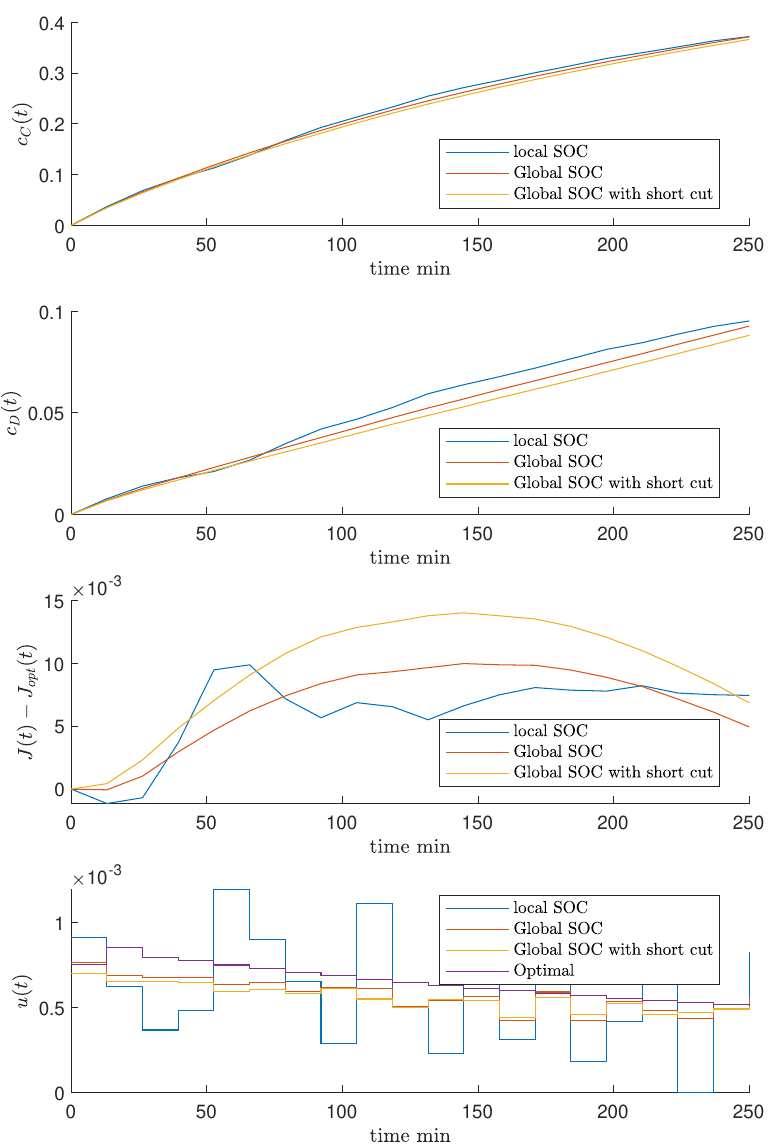} 
		\caption{$+10\%$ disturbance}
	\end{subfigure}
	\begin{subfigure}[b]{0.45\textwidth}
		\centering
		\includegraphics[width=\textwidth]{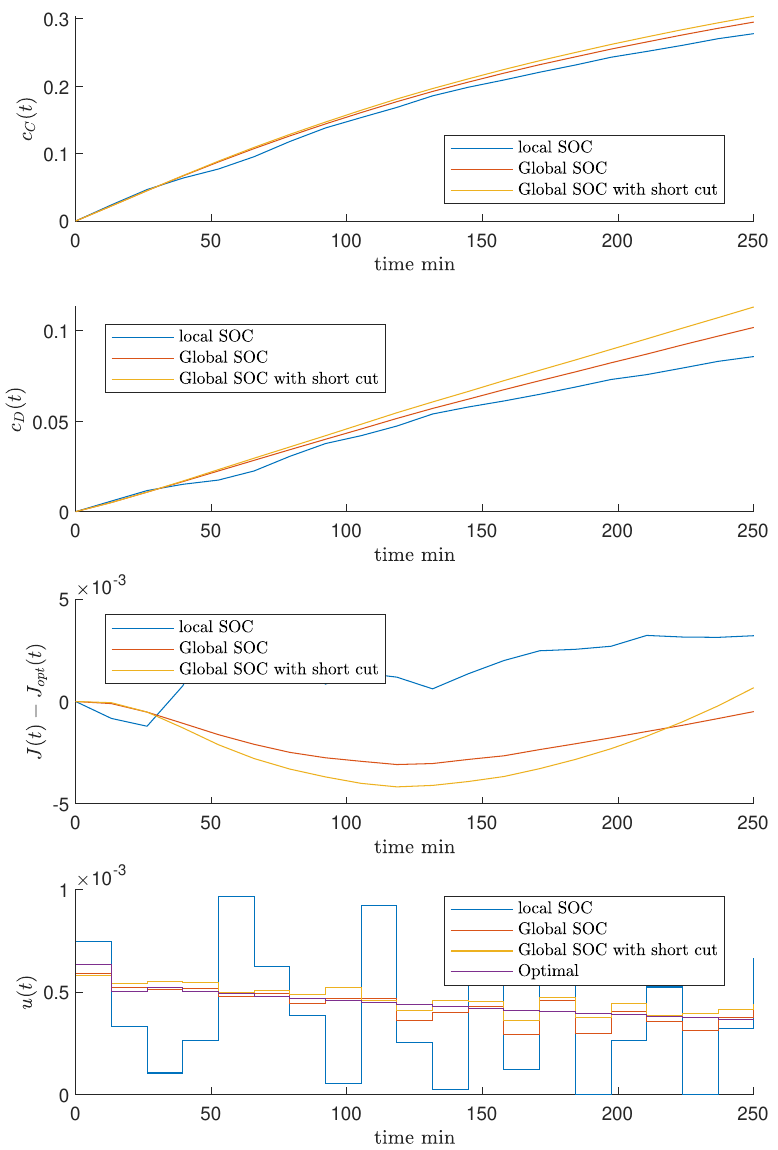} 
		\caption{$-10\%$ disturbance}
	\end{subfigure}	
	\caption{Dynamic simulation}
	\label{fig:10}	
\end{figure}
{To more intuitively demonstrate the effects of the three algorithms, we present the results of dynamic simulations with $\pm 10\%$ disturbances in Fig.~\ref{fig:10}. It can be observed that compared to the local method, the control input trajectories of the global methods are closer to the optimal input trajectories. The economic performances of the global methods at the final time step are also better than the local method.}

\begin{figure}[h]
	\includegraphics{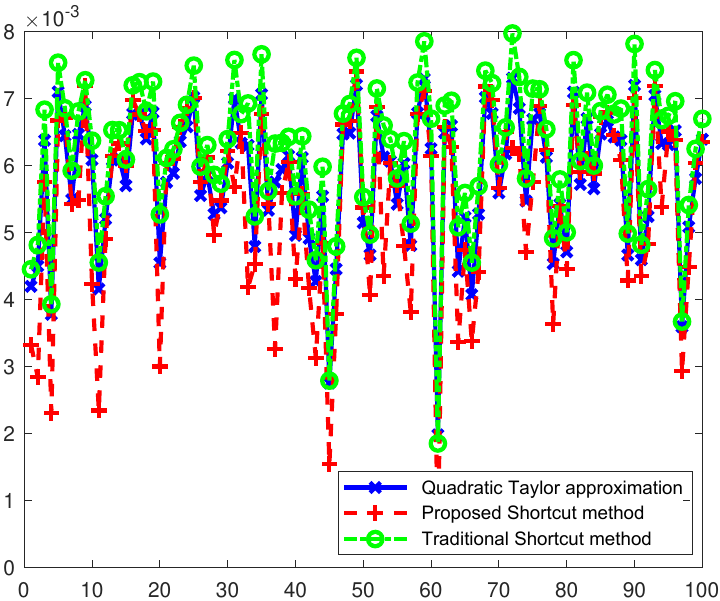} 
	\caption{Comparison of Approximation Errors of Different Approximation Methods} \label{fig:cp}		
\end{figure}

To further analyze the effectiveness of the global and shortcut methods, we contrast the approximation errors of different methods over 100 scenarios in Fig. \ref{fig:cp}.
In Fig~\ref{fig:cp}, the blue line shows the approximation error of the quadratic Taylor approximation of the economic loss which is calculated by Eq.~\ref{eq:qsloss1} at each sample.
The red line represents the approximation error of the shortcut quadratic Taylor approximation proposed in this paper which is the objective function in Eq.~\ref{eq:re_gSOC_opt_batch} . The green line denotes the approximation error obtained using traditional shortcut methods.

There are a few key observations:
\begin{enumerate}
	\item The proposed simplification (red) consistently exhibits the lowest approximation error compared to the other two methods.
	\item The traditional shortcut method (green) performs worst overall, with large errors due to its overly aggressive constraints.
\end{enumerate}
In summary, these results highlight the importance of careful, moderate simplifications when designing SOC methods for nonlinear batch processes. The proposed gdSOCsc effectively balances model accuracy and analytical tractability.

\begin{table}[h]
	\centering
	\caption{The absolute value of the Loss approximation error with 100 scenarios}
	\begin{tabular}{cccc}
		\hline
		& Local       & \multicolumn{2}{c}{Global} \\
		& {ldSOC}\cite{Ye2018} & {gdSOC} & {gdSOCsc} \\
		\hline
		average            & 0.006212   & 0.005816 & 0.005447        \\
		standard deviation & 0.001090   & 0.001024 & 0.001411       \\
		\hline
	\end{tabular}
	\label{tb:lossMSTD}
\end{table}
\begin{figure}[h]
	\centering
	\includegraphics{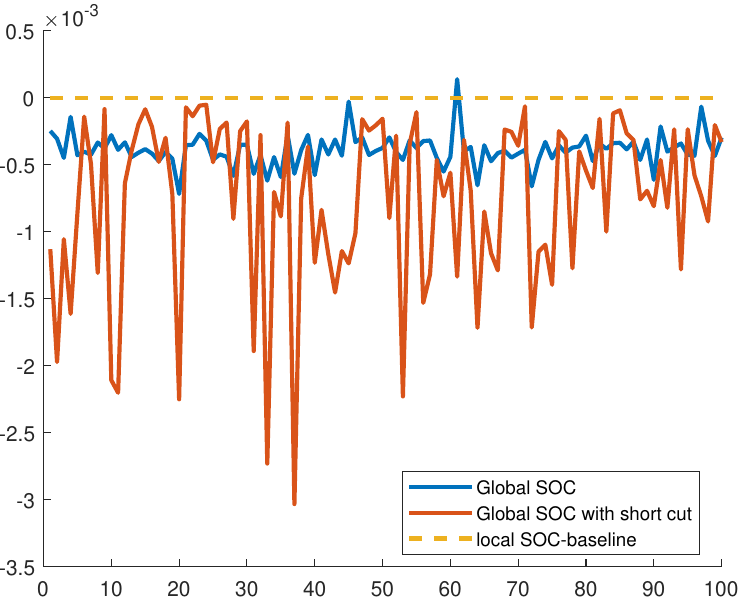} 
	\caption{Comparison of approximation error of global method and approximation error of local method}
	\label{fig:abs}
\end{figure}
To further quantify the approximation errors, Table~\ref{tb:lossMSTD} compares the absolute average and standard deviation of the errors. 
The gdSOC global method achieves a lower average error than the ldSOC local method, again showing the benefits of a global approach.
{Meanwhile, in Fig. \ref{fig:abs}, taking the local method's approximation error as a baseline, the approximation errors of the two global methods and the local method under each disturbance scenario are evaluated. The blue line in Fig. \ref{fig:abs} denotes the gdSOC approach, and the orange line signifies the gdSOCsc approach. It can be observed that the approximation errors of the two global methods are lower than those of the local method in almost all scenarios, although the differences between the global and local methods' approximation errors are minor.}
However, the gdSOCsc shortcut has the highest standard deviation, indicating significant variability in its performance across different scenarios. This implies potential issues with stability and robustness compared to the gdSOC approach.
The gdSOC strikes a balance with both low average error and standard deviation, proposing it provides the optimum combination of accuracy and reliability.

Overall, the numerical results on approximation errors provide solid evidence that:
\begin{enumerate}[]
	\item Global methods like gdSOC outperform local methods by better capturing nonlinearity.
	\item The proposed gdSOCsc shortcut approach offers computational simplicity but may sacrifice robustness.
	\item gdSOC delivers the most favorable balance between accuracy, stability and tractability.
\end{enumerate}
These findings emphasize the importance of holistic evaluation in assessing SOC techniques for batch processes. While shortcuts enable analytical solutions, moderate approximations are necessary to maintain model validity. The gdSOC method effectively navigates this trade-off for high-quality nonlinear batch process control.

\textbf{Case 2: Evaluating Uncertainty Scalability} 
In this case, we expanded the range of variation for the reaction kinetics coefficients from 5\% to 20\%.
This case aims to analyze the scalability of the gdSOC and gdSOCsc methods with broader uncertainty ranges. As noted in the background, batch processes inherently exhibit strong nonlinear dynamics across large operating spaces. It is therefore critical to evaluate if global SOC techniques can maintain performance under increasing uncertainties. {The $\mathbf{H}$ utilized in this case is consistent with case 1. }

\begin{table}[h]
	\centering
	\caption{The absolute value of the Loss approximation error with 100 scenarios}
	\begin{tabular}{cccc}
		\hline
		variation & ldSOC & {gdSOC} & {gdSOCsc} \\
		\hline
		5\%  & 0.009996 &  0.006835 & 0.006842 \\
		10\% & 0.010177  &0.007304 & 0.007426 \\
		20\% & 0.010855 &0.007735 & 0.007821 \\
		40\% & 0.013546 & 0.011245 & 0.012125 \\
		\hline
	\end{tabular}
	\label{tb:lossdd}
\end{table}
The results in Table \ref{tb:lossdd} indicate that larger uncertainty ranges lead to progressively higher economic losses for ldSOC, gdSOC and gdSOCsc, as expected. Specifically, we can see that gdSOC and gdSOCsc outperform ldSOC across all variation levels, demonstrating their improved robustness in handling uncertainties. The difference in performance becomes more pronounced as the variation level increases - at 40\% variation, gdSOC and gdSOCsc have approx. 20\% lower loss than ldSOC. Between gdSOC and gdSOCsc, the differences are minor, implying the comparable  performance between gdSOC and gdSOCsc. Overall, the results validate the superiority of the proposed gdSOC and gdSOCsc over conventional ldSOC, especially for problems with considerable uncertainties.


\textbf{Case 3:  Simplicity versus Robustness}
In previous research on batch process self-optimizing control\cite{Ye2018,Ye2022b}, it is believed that Structure 1 performs better than Structure 2, and Structure 2 performs better than Structure 3 in terms of economic performances. This is easily understandable as batch processes exhibit time-varying characteristics, and designing control strategies that correspond to these variations is natural. However, in practice, switching between control structures can introduce additional fluctuations. Additionally, if unmodeled dynamic are considered, Structure 3 might exhibit better robustness.
This case explores if using simpler, time-invariant CV structures can improve robustness to unmodeled disturbances, compared to more complex time-varying CVs. 

In this case, we use data with initial state and {kinetic coefficient} fluctuations of $ \pm 10\% $ to design the controlled variables. The linear combination matrices for the three structures are obtained through the gdSOCsc method.{ The results are as follows:}
$$
\mathbf{\bar{H}}_{s1} =  \mathrm{LBT}[\mathbf{H}_{s1}(0), \mathbf{H}_{s1}(1), \ldots, \mathbf{H}_{s1}(N-1)],
$$
where
$$
\begin{aligned}
	&\mathbf{H}_{s1}(0) = [-6.8228\times 10^{-4} \  1.6249\times 10^{-3} \  -1.9691\times 10^{-4} \  1.6716\times 10^{-4} \  -1] \\
	&\mathbf{H}_{s1}(1) =[-2.7810\	5.1044\	0.60853\	1.3657\ 	3.8392\	0.58668\	1.3982\	-1\times 10^{4}] \times 10^{-4}\\
	&...
\end{aligned}
$$
$$\mathbf{\bar{H}}_{s2} =[\bar{{\mathbf{c}}}_s,\ \operatorname{diag}[\mathbf{H}_{0,s2}(0), \ldots, \mathbf{H}_{N-1,s2}(N-1)]] $$
where 
$$
\begin{aligned}
	&c_{s2,0} =-6.7069 \times 10^{-4} ,\ \mathbf{H}_{0,s2}(0) = [	16.123\	-1.9709\	1.6506\	-1\times 10^{4}]\times 10^{-4}  \\
	&c_{s2,0} =-0.29678\times 10^{-4} ,\ \mathbf{H}_{1,s2}(1) = [ 7.1019\	0.84467\	1.7298\	-1\times 10^{4}]\times 10^{-4} \\
	&...
\end{aligned}
$$
$$
  \mathbf{H}_{s3} = [\  5.9681\times 10^{-4} \  8.3775\times 10^{-5} \  1.4724\times 10^{-4} \  -1],\
c_{s3} = 6.6631\times 10^{-5}
$$
where subscripts $s1,\ s2,\ s3$ denote the different structures.

We consider testing the performance of the controlled variables in two scenarios:
\begin{enumerate}
	\item Larger uncertainties: there is a $ \pm 20\% $ fluctuation range of initial state and kinetic parameters in testing scenarios, which indicates that the actual range of uncertainty is larger than considered.
	\item Unaccounted uncertainty: \(c_B^{in}\) is treated as an unaccounted uncertainty .
	Its range of variation is $ \pm 20\% $ around the nominal point.
\end{enumerate}

\begin{table}[h]
	\begin{tabular}{cccc}
		\hline
		& Structure 1 & Structure 2 & Structure 3 \\
		\hline
		greater disturbance     & 0.007945 & 0.007832    & 0.007959    \\
		unaccounted disturbance & 0.007194  & 0.007082    & 0.007111   \\
		\hline
	\end{tabular}
	\caption{The economic loss of controlled variables with different structures under unexpected disturbance}
	\label{tab:unexpectD}
\end{table}
Table \ref{tab:unexpectD} shows the economic loss of three CV structures under greater and unaccounted disturbances. We can see that Structure 2 achieves the lowest loss under both disturbance scenarios. This indicates Structure 2 is the most robust design among the three structures.
However, it should be noted that Structure 3 uses the simplest fixed CVs, yet its performance is comparable to Structure 2. In terms of practical implementation difficulty, Structure 3 would be the easiest to deploy. This suggests there are merits to adopting the simplest Structure 3, as it can achieve adequate robustness while requiring minimal implementation effort.
The key takeaway is that pursuing complex CV tuning does not necessarily improve robustness against uncertainties. Structures 1 and 2 use more involved time-varying CVs based on nominal conditions, yet their advantage over the basic Structure 3 is marginal.

\section{Conclusions and discussions}
Batch processes exhibit highly dynamic and nonlinear behavior due to their non-steady-state operation. This poses significant challenges for optimization and control. Traditional local SOC methods derive linear models through local analysis, which fail to adequately capture batch process nonlinearities. As a result, local SOC techniques perform poorly when applied to batch processes.

To address these limitations, this paper makes the following key contributions:
\begin{enumerate}
	\item The gSOC methodology is extended to batch processes for the first time. The new gdSOC evaluates controlled variables over the entire batch operating space using nonlinear models. This avoids the approximations of local linearization. 
	\item A proof is given that certain structural constraints on the controlled variable combination matrix can be expressed as linear constraints. This enables simplification of the gdSOC optimization problem.
	\item A novel convex approximation algorithm is proposed to solve gdSOC for batch processes with structural constraints. This approach balances model accuracy and analytical tractability.
\end{enumerate}

Specifically, we retain the nonlinear process model in gdSOC to evaluate controlled variables. We also represent the combination matrix in a vectorized form to enable linear representation of structural constraints. Additionally, we develop a new moderate approximation technique to obtain analytical solutions while preserving the nonlinear model. This paper demonstrates the first application of gdSOC to optimize batch process operations.

This work has explored several aspects of the gdSOC method for batch processes. However, there are still relevant problems that remain unaddressed but could be critical for practical implementation. Future research could focus on the following topics:
\begin{enumerate}
	\item \textbf{General nonlinear combinations of measurements as CVs} 
	Batch processes exhibit strong nonlinearity, and representing controlled variables as linear combinations of measurement variables may limit the potential for improving controlled variable performance. Existing work has introduced techniques such as neural networks and polynomials to design controlled variables \cite{su2022}.
	\item \textbf{Constrained Problem} The current discussion has primarily focused on unconstrained batch process optimization problems. However, in practice, batch processes often have various constraints, including terminal constraints and path constraints. The NCO approximation is one of the self-optimizing methods but can only be applied to scenarios without actively changing constraints \cite{Ye2013a}. In the case of batch processes, switching active constraints is common, so there is a pressing need to develop approaches that can address batch process optimization problems with constraints. Future research can explore methodologies that handle the challenges posed by constrained batch process optimization.
	\item \textbf{Structural uncertainties} 
	{The current algorithm relies on the implicit assumption that the nominal model structure correctly captures the underlying dynamics. We acknowledge that structural uncertainties could degrade the performance of our algorithm if the model provides a poor approximation of the true system. Designing an adaptive algorithm that continuously updates the model structure online based on observed data could be an interesting avenue for future work. Potential approaches could involve integrating reinforcement learning, iterative learning control, or other adaptive techniques to modify the model structure over time based on experience and observed data. This could impact the algorithm with robustness against structural model mismatches. Exploring adaptive model learning represents an intriguing direction to relax the reliance on an accurate initial system model.}
\end{enumerate}

In conclusion, there are promising opportunities for further research in exploring the application of modern regression methods for designing controlled variables in batch processes and developing approaches to tackle the challenges presented by constrained batch process optimization problems. Addressing these areas can significantly contribute to the practical implementation and performance improvement of the gdSOC method for batch processes.


\begin{acknowledgement}
This work is supported by the Research Funds of Institute of Zhejiang University-Quzhou.

\end{acknowledgement}

%
%
%

\bibliography{reference}

\end{document}